\documentclass[12pt]{article}
\usepackage{diagbox}
\usepackage{amssymb}
\usepackage{amsmath,bm}
\usepackage{multirow}
\usepackage{makecell}
\usepackage{graphicx}
\usepackage{subfigure}
\usepackage{float}
\usepackage[titletoc, title]{appendix}
\usepackage{stmaryrd}
\usepackage{lipsum}
\usepackage{amsfonts}
\usepackage{graphicx}
\usepackage{epstopdf}
\usepackage{algorithmic}
\usepackage{amsopn}
\usepackage[nameinlink]{cleveref}
\usepackage{amssymb,amsmath,amscd}
\usepackage[margin=1in]{geometry}
\usepackage{color}

\newtheorem{corollary}{Corollary}

\newtheorem{theorem}{Theorem}
\newtheorem{lemma}{Lemma}
\newtheorem{proof}{Proof}

\begin{document}

	\title{A hybridizable discontinuous Galerkin method for the quad-curl problem}
\date{}
\author{
	Gang Chen\thanks{School of Mathematics Sciences, University of Electronic Science and Technology of China, Chengdu 611731, China. email: cglwdm@uestc.edu.cn. The first author's research is supported by
		\color{black} National Natural Science Foundation of China (NSFC) grant no. 11801063, China Postdoctoral Science Foundation grant no. 2018M633339, and Key Laboratory of Numerical Simulation of Sichuan Province (Neijiang, Sichuan Province) grant no. 2017KF003.}, 
	Jintao Cui\thanks{Department of Applied Mathematics, The Hong Kong Polytechnic University,
		Hung Hom, Hong Kong.
		email:jintao.cui@polyu.edu.hk.
		The second author's research is supported in part by \color{black} the Hong Kong RGC, General Research Fund (GRF) grant no. 15302518 and the National Natural Science Foundation of China (NSFC) grant no. 11771367.}, 
	Liwei Xu\thanks{School of Mathematics Sciences,
		University of Electronic Science and Technology of China,
		Chengdu 611731, China.  Corresponding author: email:xul@uestc.edu.cn. The third author's research is supported in part by a Key Project of the Major Research Plan of NSFC grant no. 91630205 and the NSFC grant no. 11771068.}
	}
\maketitle
\begin{abstract} 
	The quad-curl problem arises in magnetohydrodynamics, inverse electromagnetic scattering and transform eigenvalue problems. In this paper we investigate a hybridizable discontinuous Galerkin method to solve the quad-curl problem based on a mixed formulation. The divergence-free condition is enforced by introducing a Lagrange multiplier into the system. The analysis is performed for the model problem with low regularity, which is posed on a Lipschitz polyhedron domain.

keywords: quad-curl, HDG method, low regularity
\end{abstract}

	\section{Introduction.}
\label{456}
Let $\Omega$ be a bounded simply-connected Lipschitz polyhedron in $\mathbb{R}^3$
with connected boundary $\Gamma$. We consider the following mixed quad-{\rm curl} problem:
\par\smallskip\noindent
Find the vector field $\bm{u}$ such that
\begin{eqnarray}
\left\{
\begin{aligned}
\nabla\times\nabla\times\nabla\times\nabla\times \bm{u}
+\nabla p&=\bm{f}\,\,\,\, \text{ in }\Omega, \label{source}\\
\nabla\cdot\bm{u}&=g\,\,\,\,\,  \text{ in }\Omega,  \\
\bm{n}_{\Gamma}\times \bm{u} &=\bm{g}_{1}\,\, \text{ on }\Gamma,\\
\bm{n}_{\Gamma}\times \nabla\times\bm{u} &=\bm{g}_{2}\,\, \text{ on }\Gamma,\\
p&=0\,\,\,\,\, \text{ on }\Gamma.
\end{aligned}
\right.
\end{eqnarray}
Here $\bm{n}_{\Gamma}$ is the outward normal unit vector to the domain boundary $\Gamma$, $\bm{f}\in [L^2(\Omega)]^3$
is an external source filed, $g\in L^2(\Omega)$ and $\bm{g}_1,\bm g_2\in \bm{H}^{-\frac{1}{2}}({\rm div}_{\tau};\Gamma)\cap[H^{\delta}(\Gamma)]^3$ 
are given functions with $\delta>0$, where $\bm{H}^{-\frac{1}{2}}({\rm div}_{\tau};\Gamma)$ is the range space of ``tangential trace" of space $\bm{H}({\rm curl};\Omega)$. See \cite{Buffa} for the detailed description of space $\bm{H}^{-\frac{1}{2}}({\rm div}_{\tau};\Gamma)$.
\par

The model problem \eqref{source} arises in many areas such as magnetohydrodynamics, inverse electromagnetic scattering and transform eigenvalue problems. The main challenges of designing accurate, robust and efficient numerical methods for \eqref{source} are listed as follows.
\begin{itemize}
	\item  The construction of $H^2$-conforming and {\rm curl}-{\rm curl}-conforming finite elements for solving the quad-curl problem would be very complicated.
	\item The quad-{\rm curl} operator is not positive definite; hence it is difficult to design suitable numerical schemes. Moreover, it makes the error analysis and the design of fast solvers more complicated.
	\item The regularity of the model problem \eqref{source} on general polyhedral domain is still unknown. The existing work on numerical schemes are all based on high regularity assumptions.
	\item When Lagrange multiplier is introduced to enforce the divergence-free condition, an inf-sup condition is required in order to ensure existence and uniqueness of the approximation of $p$.
\end{itemize}
\par
There is only few work on numerical methods for quad-{\rm curl} problem on three-dimensional domains. In \cite{Xu1}, a nonconforming finite element method for the quad-{\rm curl} model with low order term was studied under the regularity assumption such that
\begin{eqnarray}
\bm u \in [H^4(\Omega)]^3. \label{H4}
\end{eqnarray}
A discontinuous Galerkin (DG) method using $H({\rm curl})$ conforming elements for the quad-{\rm curl} model problem was investigated in \cite{Xu2},
where the following lower regularity requirement was considered:
\begin{eqnarray} \label{H2Assmp}
\bm u \in [H^2(\Omega)]^3,\, \nabla\times\bm u\in [H^2(\Omega)]^3.
\end{eqnarray}
A mixed FEM for the quad-{\rm curl} eigenvalue problem was introduced and analyzed in \cite{Sun} under the regularity assumption higher than \eqref{H2Assmp}, such that
\begin{eqnarray}
\bm u \in [H^3(\Omega)]^3,\, \nabla\times\bm u\in [H^3(\Omega)]^3.
\end{eqnarray}
The quad-{\rm curl} problem in 2D was studied in \cite{Susan} based on Hodge decomposition.
\par

Some regularity results of the quad-curl problem on domains with particular geometries also exists in the literature. For example, the following results were proved in \cite{Ni}: when $\bm f\in [L^2(\Omega)]^3$, $\bm g_1=\bm g_2=\bm 0$, $g=0$, and the domain has no point and edge singularities, it holds that $\bm u\in [H^4(\Omega)]^3$; when the domain has point and edge singularities, $\bm u$ does not belong to $[H^3(\Omega)]^3$ in general. In \cite{Zhang2018Mixed}, the author proved that on convex polyhedral domains, when $\bm g_1=\bm g_2=\bm 0$, $g=0$ and $\nabla\cdot\bm f=0$, there holds
\begin{align}
\bm u\in[H^2(\Omega)]^3,\, \nabla\times\bm u\in [H^2(\Omega)]^3,\, p=0.
\end{align}
However, there are no regularity results available for general Lipschitz polyhedral domains.
\par
In recent years, the hybridizable discontinuous Galerkin (HDG) method has been successfully applied to solve various types of differential equations. It retains the main advantages of standard DG methods, such as flexible in meshing, easy to design and implement, ideal to be used with hp-adaptive strategy, etc. Moreover, HDG method can significant reduce the number of degrees of freedom, which allows for a substantial reduction in the computational cost. In this paper, we propose and analyze a HDG method for quad-curl model problem \eqref{source}, aiming to tackle the difficulties mentioned above. The error analysis is based on the following lower regularity requirement:
\begin{align}
\bm u\in[H^s(\Omega)]^3,\, \nabla\times\bm u\in [H^{s+1}(\Omega)]^3,\,  \nabla\times\nabla\times\bm u\in [H^s(\Omega)]^3,\, p\in H^{1+s}(\Omega),
\end{align}
with $s \in \left(\frac{1}{2}, 1\right]$. Actually, such regularity results hold for simply-connected Lipschitz polyhedron.
\par
The rest of this paper is organized as follows. In section 2, we give some preliminaries, including basic notations, the regularity results based a mixed formulation and several projection operators needed for error analysis. In section 3, we propose the HDG method for the quad-curl model problem and show its stability results. In section 4, we derive the convergence analysis of the proposed HDG scheme. In section 5, some numerical experiments are performed to verify our theoretical results.
\par
Throughout this paper, we use $C$ to denote a positive constant independent of mesh size and the frequency $w$, not necessarily the same at its each occurrence.   We  use $a\lesssim b$ ($a\gtrsim b$)
to represent $a\le Cb$ ($a\ge Cb$), and $a\sim b$ to represent $a\lesssim b\lesssim a$.

\section{Preliminaries}

\subsection{Notations}
For any bounded domain $\Lambda\color{black}\subset\color{black} \mathbb{R}^s$ $(s=2,3)$, let $H^{m}(\Lambda)$ and $H^m_0(\Lambda)$  denote the usual  $m^{th}$-order Sobolev spaces on $\Lambda$, and $\|\cdot\|_{m, \Lambda}$, $|\cdot|_{m,\Lambda}$  denote the norm and semi-norm on these spaces, respectively. We use $(\cdot,\cdot)_{m,\Lambda}$ to denote the inner product of $H^m(\Lambda)$, with $(\cdot,\cdot)_{\Lambda}:=(\cdot,\cdot)_{0,\Lambda}$.
When $\Lambda=\Omega$, we denote $\|\cdot\|_{m }:=\|\cdot\|_{m, \Omega}$, $|\cdot|_{m}:=|\cdot|_{m,\Omega}$ and $(\cdot,\cdot):=(\cdot,\cdot)_{\Omega}$. In particular, when $\Lambda\in \mathbb{R}^{2}$, we use $\langle\cdot,\cdot\rangle_{\Lambda}$ to replace $(\cdot,\cdot)_{\Lambda}$; when $\Lambda\in \mathbb{R}^{1}$, we use $\langle\!\langle\cdot,\cdot\rangle\!\rangle_{\Lambda}$ to replace $(\cdot,\cdot)_{\Lambda}$ to make a distinction. The bold face fonts will be used for vectors (or tensors) analogues of the Sobolev spaces along with vector-valued (or tensor-valued) functions. For integer $k\geq 0$, we denote by $\mathcal{P}_{k}(\Lambda)$ the set of polynomials defined on $\Lambda$ with degree no greater than $k$.
\par
Let $\mathcal{T}_h=\bigcup\{T\}$  be a shape regular simplex partition of the domain $\Omega$ consists of arbitrary polygons. For any $T\in\mathcal{T}_h$, we let $h_T$ be the infimum of the diameters of spheres containing $T$ and denote the mesh size $h:=\max_{T\in\mathcal{T}_h}h_T$. Let $\mathcal{F}_h=\bigcup\{F\}$ be the union of all faces of $T\in\mathcal{T}_h$, and let $\mathcal{F}_h^o$ and $\mathcal{F}_h^B$ be the set of interior faces and boundary faces, respectively. We denote by $h_F$ the length of diameter of the smallest circle containing face $F$.
For all  $T\in\mathcal{T}_h$ and $F\in\mathcal{F}_h$, we denote by $ \bm{n}_T $ and $\bm{n}_F$   the  unit outward  normal vectors along $\partial T$ and  face $F$, respectively. Broken curl-curl, \text{{\rm curl}}, \text{{\rm div}} and gradient operators with respect to
decomposition $\mathcal{T}_h$ are donated by $\nabla\times\nabla\times$, $\nabla \times$, $\nabla \cdot$ and $\nabla $, respectively.
For $u,v\in L^2(\partial\mathcal{T}_h)$, we define the following inner product and the corresponding norm:
\begin{align*}
\langle u,v \rangle_{\partial\mathcal{T}_h}=\sum_{T\in\mathcal{T}_h}\langle u,v\rangle_{\partial T},
\qquad
\|v\|^2_{\mathcal{T}_h}=\sum_{T\in\mathcal{T}_h}\|v\|^2_{0,T},\qquad
\|v\|^2_{\partial\mathcal{T}_h}=\sum_{T\in\mathcal{T}_h}\|v\|^2_{0,\partial T}.
\end{align*}

Define the following function spaces
\begin{align*}
\bm{H}(\text{{\rm curl}};\Omega)&:=\big\{\bm{v}\in [L^2(\Omega)]^3: \nabla\times\bm{v}\in [L^2(\Omega)]^3 \big\},\\
\bm{H}^s(\text{{\rm curl}};\Omega)&:=\big\{\bm{v}\in [H^s(\Omega)]^3: \nabla\times\bm{v}\in [H^s(\Omega)]^3 \big\}\,\text{ with }s\ge 0,\\
\bm{H}_0(\text{{\rm curl}};\Omega)&:=\big\{\bm{v}\in \bm{H}(\text{{\rm curl}};\Omega): \bm{n}_{\Gamma}\times\bm{v}|_{\Gamma}=\bm{0}\big \},\\
\bm{H}(\text{{\rm div}} ;\Omega)&:=\big\{\bm{v}\in [L^2(\Omega)]^3: \nabla\cdot\bm{v}\in L^2(\Omega)\big \},\\
\bm{H}_0(\text{{\rm div}} ;\Omega)&:=\big\{\bm{v}\in \bm{H}(\text{{\rm div}} ;\Omega): \bm{n}\cdot\bm{v}=0 \big\},\\
\bm{H}(\text{{\rm div}} ^0;\Omega)&:=\big\{\bm{v}\in \bm{H}(\text{{\rm div}} ;\Omega): \nabla\cdot\bm{v}=0\big\},
\end{align*}
and
\begin{align*}
\bm{X} &:=\bm{H}(\text{{\rm curl}};\Omega)\cap \bm{H}(\text{{\rm div}} ;\Omega),\\
\bm{X} _N&:=\bm{H}_0(\text{{\rm curl}};\Omega)\cap \bm{H}(\text{{\rm div}} ;\Omega),\\
\bm{X} _T&:=\bm{H}(\text{{\rm curl}};\Omega)\cap \bm{H}_0(\text{{\rm div}} ;\Omega).
\end{align*}
We define the following norm on $\bm{H}^s(\text{{\rm curl}};\Omega)$ with $s\ge 0$:
\begin{eqnarray*}
	\|\bm{v}\|_{s,\text{{\rm curl}}}=\left(  \|\bm{v}\|_s^2+\|\nabla\times\bm{v}\|_s^2    \right)^{\frac{1}{2}}.
\end{eqnarray*}

\subsection{Regularity}
By introducing $\bm{r}=\nabla\times\nabla\times\bm{u}$ we can rewrite \eqref{source} into the following second order system.
\par\smallskip\noindent
Find $(\bm{r},\bm{u},p)$ that satisfies
\begin{eqnarray}
\left\{
\begin{aligned}
\bm{r}-\nabla\times\nabla\times\bm{u}&=\bm{0}\,\,\,\,\,\text{ in }\Omega,\\
\nabla\times\nabla\times\bm{r}+\nabla p&=\bm{f}\,\,\,\, \text{ in }\Omega, \\
\nabla\cdot\bm{u}&=g\,\,\,\,\,  \text{ in }\Omega,  \\
\bm{n}_{\Gamma}\times \bm{u} &=\bm{g}_1\,\, \text{ on }\Gamma,\\
\bm{n}_{\Gamma}\times\nabla\times \bm{u} &=\bm{g}_2\,\, \text{ on }\Gamma, \label{mix0source}\\
p &=0\,\,\,\, \text{ on } \Gamma.
\end{aligned}
\right.
\end{eqnarray}
We assume that the following regularity holds true:
\begin{align}
\bm r\in \bm{H}^s(\Omega),\, \bm u\in \bm{H}^s(\text{{\rm curl}};\Omega),\, \nabla\times\bm u\in \bm{H}^{1+s}(\text{{\rm curl}};\Omega),\,\text{and}\,\, p\in H^{1+s}(\Omega), \label{regular}
\end{align}
with $s\in \left(\frac{1}{2},1\right]$
The designing of HDG scheme  will based on
equations \eqref{mix0source}, and the analysis of HDG scheme will based on the regularity \eqref{regular}.

\subsection{Projection operators}

\subsubsection{$L^2$-projection}
For any $T\in\mathcal{T}_h$, $ F\in\mathcal{F}_h$ and integer $j\ge 0$, let $\Pi_j^o: L^2(T)\rightarrow \mathcal{P}_{j}(T)$ and $\Pi_j^{\partial}: L^2(F)\rightarrow \mathcal{P}_{j}(F)$ be the usual $L^2$ projection operators.
The following approximation and stability results are standard.
\begin{lemma}\label{lemma2.5}
	For any $T\in\mathcal{T}_h$ and $F\in\mathcal{F}_h$ \color{black} and nonnegative integer $j$\color{black}, it holds
	\begin{align*}
	\|v-\Pi_j^ov\|_{0,T}&\lesssim h_T^{s}|v|_{s,T}\,\,\,\quad\quad\forall\, v\in H^{s}(T),\\
	\|v-\Pi_j^ov\|_{0,\partial T}&\lesssim h_T^{s-1/2}|v|_{s,T}\quad\forall\, v\in H^{s}(T),\\
	\|v-\Pi_j^{\partial}v\|_{0,\partial T}&\lesssim h_T^{s-1/2}|v|_{s,T}\quad\forall\, v\in H^{s}(T),\\
	\|\Pi_j^ov\|_{0,T}&\le\|v\|_{0,T}\quad\quad\quad\,\,\forall\, v\in L^{2}(T),\\
	\|\Pi_j^{\partial}v\|_{0,F}&\le\|v\|_{0,F}\,\,\quad\quad\quad\forall\, v\in L^{2}(F),
	\end{align*}
	where $1/2< s\le j+1$.
\end{lemma}

\subsubsection{$H({\rm div})$-projection}
For integer $j\ge 1$, we first define the following local spaces on $T$ and $F$.
\begin{align*}
\bm{\mathcal{D}}_{j}(T)=[\mathcal{P}_{j-1}(T)]^3\oplus\widetilde{\mathcal{P}}_{j-1}(T)\cdot\bm x,\\
\bm{\mathcal{D}}_{j}(F)=[\mathcal{P}_{j-1}(F)]^3\oplus\widetilde{\mathcal{P}}_{j-1}(F)\cdot\bm x,
\end{align*}
where $\widetilde{\mathcal P}_{j-1}$ denotes the subspace of $\mathcal{P}_{j-1}$ consisting of homogeneous polynomials of degree $j-1$. Then we define the global $H({\rm div})$ space by
\begin{align*}
\bm X_{h,j}=\{\bm v_h\in \bm H({\rm div};\Omega): \bm v_h|_T\in \bm{\mathcal{D}}_{j}(T), \ \forall\, T\in\mathcal{T}_h\}.
\end{align*}
For any $\bm v\in\bm H({\rm div};\Omega)$, its $H({\rm div})$-projection $\bm{\Pi}^{{\rm div}}_{h,j}\bm v\in \bm X_{h,j}$ is defined as follows (see Ref. \cite{Ned1}).
\begin{subequations}
	\begin{align}
	\langle\bm{\Pi}^{{\rm div}}_{h,j}\bm v\cdot\bm n, w_{j-1}\rangle_F
	&=\langle \bm v\cdot\bm  n, w_{j-1}\rangle_F\,\,\,\, \forall\,  w_{j-1}\in\mathcal{P}_{j-1}(F),\label{or-div-F} \\
	(\bm{\Pi}^{{\rm div}}_{h,j}\bm v,\bm w_{j-2})_T
	&=(\bm v,\bm w_{j-2})_T\quad\quad \forall\, \bm w_{j-2}\in [\mathcal{P}_{k-2}(T)]^3,\label{or-div-T}
	\end{align}
\end{subequations}
which hold for all $T\in\mathcal{T}_h$, $F\subset\partial T$ and $E\subset\partial F$.

\subsubsection{$H({\rm curl})$-projection}
For integer $j\ge 1$, we define
\begin{align*}
\bm{Y}_{h,j}&=\{\bm v_h\in \bm{H}({\rm curl};\Omega):\bm v_h|_T\in [\mathcal{P}_{j}(T)]^3,\ \forall\, T\in\mathcal{T}_h \},
\end{align*}
For any $\bm v\in\bm H^{s}({\rm curl};\Omega)$ ($s>\frac{1}{2}$), its H(curl)-projection $\bm{\Pi}^{{\rm curl}}_{h,j}\bm v\in \bm Y_{h,j}$ is defined as follows (see Ref. \cite{Ned2} for details).
\begin{subequations}
	\begin{align}
	\langle\!\langle\bm{\Pi}^{{\rm curl}}_{h,j}\bm v\cdot\bm \tau, w_{j}\rangle\!\rangle_E
	&=\langle\!\langle\bm v\cdot\bm  \tau, w_{j-1}\rangle\!\rangle_E\,\,\, \forall\,  w_{j}\in\mathcal{P}_{j}(E),\label{or-curl-E} \\
	\langle\bm{\Pi}^{{\rm curl}}_{h,j}\bm v,\bm w_{j-1}\rangle_F
	&=\langle\bm v\cdot,\bm w_{j-1}\rangle_F\qquad\, \forall\,  w_{j-1}\in\bm{\mathcal{D}}_{j-1}(F),\label{or-curl-F} \\
	(\bm{\Pi}^{{\rm curl}}_{h,j}\bm v,\bm w_{j-2})_T
	&=(\bm v,\bm w_{j-2})_T\qquad\,\, \forall\, \bm w_{j-2}\in \bm{\mathcal{D}}_{k-2}(T),\label{or-curl-T}
	\end{align}
\end{subequations}
which hold for all $T\in\mathcal{T}_h$ and $F\subset\partial T$.
\par
Note that for each $T\in\mathcal{T}_h$, the above projection makes sense for $\bm{v}\in \bm{H}^s({\rm curl};T)$ with $s>\frac{1}{2}$ (see \cite[Lemma 5.1]{Regularity} for details).
Moreover, the following approximation properties hold true:
\begin{lemma}\label{lem:21} \cite{Ned2,Regularity,Monk}
	For any $T\in\mathcal{T}_h$ and $\bm{v}\in \bm{H}^s({\rm curl};T)$ with $s>\frac{1}{2}$,
	if $\bm{v}\in [H^t(T)]^3$ with $t\in (\frac{1}{2},k+1]$, it holds that
	\begin{eqnarray}
	\|\bm{v}-\bm{\Pi}_{h,k}^{{\rm curl}}\bm{v}\|_{{\mathcal T_h}}\lesssim h^t\|\bm{v}\|_{t}.
	\end{eqnarray}
	Moreover, if $\nabla\times\bm{v}\in [H^t(T)]^3$ with $t\in (\frac{1}{2},k]$, then it holds
	\begin{eqnarray}
	\|\nabla\times\bm{v}-\nabla\times\bm{\Pi}_{h,k}^{{\rm curl}}\bm{v}\|_{{\mathcal T_h}}\lesssim h^t\|\nabla\times\bm{v}\|_{t}.
	\end{eqnarray}
\end{lemma}
\begin{lemma} For any integer $j\ge 1$, we have the following commuting property
	\begin{align}
	\nabla\times\bm{\Pi}^{{\rm curl}}_{h,j}\bm{v}&=\bm{\Pi}^{{\rm div}}_{h,j}\nabla\times\bm v\,\, \forall\, \bm v\in \bm{H}^s({\rm curl};\Omega),\, s>\frac{1}{2}.
	\end{align}
\end{lemma}

\subsubsection{H({\rm div})-projection on domain surface}
For any $\bm{v}\in \bm{H}^s({\rm curl};T)$ with $s>\frac{1}{2}$ and $T\in\mathcal{T}_h$, we consider $\bm{\Pi}_{h,k}^{{\rm curl}}\bm{v}$ restricted to face $F\subset \partial T$ such that
\begin{align}
\langle\!\langle\bm{\Pi}_{h,k}^{{\rm curl}}\bm{v}\cdot\bm{t}_{FE},w_k\rangle\!\rangle_{E}
=\langle\!\langle\bm{v}\cdot\bm{t}_{FE},w_k\rangle\!\rangle_{E}\quad \forall\, w_k\in \mathbb{P}_k(E),\, E\subset\partial F,\label{314}
\end{align}
where $\bm{t}_{FE}=\bm{n}_F\times\bm{n}_{FE}$. 
It can be observed that for $k\ge 2$,
\begin{eqnarray}
\langle \bm{\Pi}_{h,k}^{{\rm curl}}\bm{v},\bm{w}_{k-1}  \rangle_{F}=\langle \bm{v},\bm{w}_{k-1}  \rangle_{F}\quad\forall \,\bm{w}_{k-1}\in \bm{\mathcal{D}}_{k-1}(F).\label{315}
\end{eqnarray}

\par
Note that
\begin{equation*}
\bm{v}\cdot\bm{t}_{FE}=\bm{v}\cdot(\bm{n}_{F}\times\bm{n}_{FE} )
= (\bm{v}\times\bm{n}_{F}) \cdot  \bm{n}_{E}
=-( \bm{n}_{F}\times\bm{v})\cdot  \bm{n}_{E} ,
\end{equation*}
and
\begin{align*}
\bm{v}|_F&=(\bm{n}_F\times\bm{v})\times\bm{n}_F+(\bm{v}\cdot\bm{n}_F)\bm{n}_F,
&\bm{v}\cdot\bm{n}_F=0.
\end{align*}
Hence we can rewrite equations \eqref{314} and \eqref{315} as
\begin{align*}
\langle\!\langle\bm{\Pi}_{h,k}^{{\rm curl}}\cdot( \bm{n}_{F}\times\bm{v})\cdot  \bm{n}_{E},w_k\rangle\!\rangle_{E}
=\langle\!\langle( \bm{n}_{F}\times\bm{v})\cdot \bm{n}_{E},w_k\rangle\!\rangle_{E}\quad \forall\, w_k\in \mathbb{P}_k(E),\, E\subset\partial F,
\end{align*}
and
\begin{align*}
\langle \bm{\Pi}_{h,k}^{{\rm curl}}(\bm{n}_F\times\bm{v}),\bm{n}_F\times\bm{w}_{k-1}  \rangle_{F}
=\langle \bm{n}_F\times\bm{v},\bm{n}_F\times\bm{w}_{k-1}  \rangle_{F}\quad\forall\, \bm{w}_{k-1}\in \bm{\mathcal{D}}_{k-1}(F),\, k\ge 2.
\end{align*}
The operator $\bm{\Pi}_{h,k}^{{\rm curl}}(\bm{n}_F\times\cdot)$ maps from space $\bm{H}^s({\rm curl};T)$ to $\bigcup_{F\subset\partial T}[\mathbb{P}_k(F)]^3$ for each $T\in\mathcal{T}_h$. Actually, by denoting $\bm{\Pi}_{h,k}^{\Gamma,{\rm div}}:=\bm{\Pi}_{h,k}^{{\rm curl}}|_{\Gamma}$, we observe that $\bm{\Pi}_{h,k}^{\Gamma,{\rm div}}(\bm{n}_F\times\bm{v})$ defines a $H({\rm div})$-projection of $\bm{n}_F\times\bm{v}$ on the domain surface $\Gamma$.
\subsubsection{$H({\color{black}\rm curl})$-projection and $H^1$-projection on finite element spaces}

In the error analysis, we need the following $\bm H_0({\rm curl})$-conforming and $H^1$-conforming
interpolations.

\begin{lemma}[cf. {\cite[Proposition 4.5]{MR2194528}}]\label{pic}
	For any integer $k\ge 1$, let $\bm v_h\in [\mathcal P_{k}(\mathcal{T}_h)]^3$, there exists a function $\bm{\Pi}_{h,k}^{\rm curl,c}\bm v_h\in [\mathcal P_{k}(\mathcal{T}_h)]^3\cap \bm H_0({\rm curl};\Omega)
	$ such that
	\begin{align}
	\|\bm{\Pi}_{h,k}^{\rm curl,c}\bm v_h-\bm v_h\|_0&\lesssim \|h_F^{1/2}\bm n\times[\![\bm v_h]\!]\|_{0,\mathcal{F}_h},\label{ic1}\\
	\|\nabla_h\times(\bm{\Pi}_{h,k}^{\rm curl,c}\bm v_h-\bm v_h)\|_{0}&\lesssim \|h_F^{-1/2}\bm n\times[\![\bm v_h]\!]\|_{0,\mathcal{F}_h},\label{ic2}
	\end{align}
	with a constant $C>0$ independent of the mesh size.
\end{lemma}
\begin{lemma}[ cf. {\cite[Theorem 2.2]{ick}}]  For all $q_h\in P_h$, there exists an interpolation operator $\mathcal{I}_k: P_h\to P_h\cap H^1_0(\Omega)$ such that
	\begin{align}
	\|\nabla  q_h-\nabla \mathcal{I}_k q_h\|_{_{\mathcal T_h}}\lesssim \|h_F^{-1/2}[\![q_h]\!]\|_{\mathcal{F}_h},\label{ick}
	\end{align}
	where $[\![q_h]\!]$ stands for the jump of $q_h$ on $\mathcal F_h$.
\end{lemma}

\section{HDG finite element method}

\subsection{HDG method}

For any integers $k\ge 1$,  we introduce the following finite dimensional spaces:
\begin{align*}
\bm{R}_{h}&=[\mathcal{P}_{k-1}(\mathcal{T}_h)]^3, \\
\bm{U}_{h}&= [\mathcal{P}_k(\mathcal{T}_h)]^3,\\
\widehat{\bm{U}}_{h}&=\{\widehat{\bm{v}}_h\in [\mathcal{P}_{k}(\mathcal{F}_h)\color{black} ]^3  \color{black}: \widehat{\bm{v}}_h\cdot\bm{n}|_{\mathcal{F}_h}=0\} ,\\
\widehat{\bm{U}}_{h}^{\widetilde{\bm{g}}}&=\{\widehat{\bm{v}}_h\in \widehat{\bm{U}}_{h}:
\bm{n}_{\Gamma}\times\widehat{\bm{v}}_h|_{\Gamma}=\bm{\Pi}_{h,k}^{\Gamma,{\rm div}}\widetilde{\bm{g}}  \} ,\,\widetilde{\bm{g}}=\bm{0},\bm{g}_1,\\
\widehat{\bm{C}}_{h}&=\{\widehat{\bm{v}}_h\in [\mathcal{P}_{k-1}(\mathcal{F}_h)\color{black} ]^3  \color{black}: \widehat{\bm{v}}_h\cdot\bm{n}|_{\mathcal{F}_h}=0\} ,\\
\widehat{\bm{C}}_{h}^{\widetilde{\bm{g}}}&=\{\widehat{\bm{v}}_h\in \widehat{\bm{C}}_{h}:
\bm{n}_{\Gamma}\times\widehat{\bm{v}}_h|_{\Gamma}=\bm{\Pi}_{k-1}^{\partial}\widetilde{\bm{g}}  \} ,\,\widetilde{\bm{g}}=\bm{0},\bm{g}_2,\\
{P}_{h}&=\mathcal{P}_{k}(\mathcal{T}_h),\\
\widehat{P}_h&=\mathcal{P}_{k}(\mathcal{F}_h),\\
\widehat{P}^0_{h}&=\{\widehat{q}_h\in \widehat{P}_h:\widehat{q}_h|_{\Gamma}=0\},
\end{align*}
where
\begin{align*}
\mathcal{P}_j(\mathcal{T}_h)&=\{q_h\in L^2(\Omega): q_h|_T\in \mathcal{P}_j(T),\ \forall T\in\mathcal{T}_h\},\\
\mathcal{P}_j(\mathcal{F}_h)&=\{q_h\in L^2(\mathcal{F}_h): q_h|_F\in \mathcal{P}_j(F),\ \forall F\in\mathcal{F}_h\}.
\end{align*}

The HDG finite element method for \eqref{source} reads:
\par\smallskip\noindent
For all $(\bm{s}_h,\bm{v}_h,\widehat{\bm{v}}_h,\widehat{\bm d}_h,q_h,\widehat{q}_h)\in \bm{R}_h\times\bm{U}_h\times\widehat{\bm{U}}_h^{\bm{0}}\times\widehat{\bm{C}}_h^{\bm{0}}\times {P}_h\times \widehat{{P}}^0_h$,
find $(\bm{r}_h,\bm{u}_h,\widehat{\bm{u}}_h,\widehat{\bm {c}}_h,p_h,\widehat{p}_h) \in \bm{R}_h\times\bm{U}_h\times\widehat{\bm{U}}_h^{\bm{g}_1}\times\widehat{\bm{C}}_h^{\bm{g}_2}\times {P}_h\times \widehat{{P}}^0_h$
such that
\begin{subequations}\label{HDG-curl}
	\begin{align}
	a_h(\bm{r}_h,\bm{s}_h)+b_h(\bm{u}_h,\widehat{\bm{u}}_h,\widehat{\bm c}_h;\bm{s}_h)&=0,\label{fhm1}\\
	b_h(\bm{v}_h,\widehat{\bm{v}}_h,\widehat{\bm d}_h;\bm{r}_h)+c_h(p_h,\widehat{p}_h;\bm{v}_h)-s^u_h(  \bm{u}_h,\widehat{\bm{u}}_h,\widehat{\bm u}_h; \bm{v}_h,\widehat{\bm{v}}_h,\widehat{\bm d}_h )&=-(\bm{f},\bm{v}_h),\\
	c_h(q_h,\widehat{q}_h;\bm{u}_h)+s_h^p(p_h,\widehat{p}_h;q_h,\widehat{q}_h)&=(g, q_h),\label{fhm3}
	\end{align}
\end{subequations}
where
\begin{align}
a_h(\bm{r}_h,\bm{s}_h)&=(\bm{r}_h,\bm{s}_h)_{\mathcal T_h},\label{a_h}\nonumber\\
b_h(\bm{u}_h,\widehat{\bm{u}}_h,\widehat{\bm c}_h;\bm{s}_h)&=-(\bm{u}_h,\nabla \times\nabla \times\bm{s}_h)_{\mathcal T_h}-\langle \bm{n}\times\widehat{\bm{u}}_h,\nabla \times\bm{s}_h \rangle_{\partial\mathcal{T}_h}
-\langle \bm{n}\times\widehat{\bm {c}}_h,\bm{s}_h \rangle_{\partial\mathcal{T}_h}\nonumber\nonumber\\
c_h(q_h,\widehat{q}_h;\bm{u}_h)&=(\nabla \cdot\bm{u}_h, q_h)_{\mathcal T_h}-\langle \bm{n}\cdot\bm{u}_h,\widehat{q}_h \rangle_{\partial\mathcal{T}_h},\nonumber\\
s_h^u(  \bm{u}_h,\widehat{\bm{u}}_h,\widehat{\bm c}_h; \bm{v}_h,\widehat{\bm{v}}_h,\widehat{\bm d}_h )&=\langle h_F^{-3} \bm{n}\times(\bm{u}_h-\widehat{\bm{u}}_h),\bm{n}\times(\bm{v}_h-\widehat{\bm{v}}_h)\rangle_{\partial\mathcal{T}_h}\nonumber\nonumber\\
&\qquad+\langle h_F^{-1} \bm{n}\times(\nabla \times\bm{u}_h-\widehat{\bm {c}}_h), \bm{n}\times(\nabla \times\bm{v}_h-\widehat{\bm d}_h)\rangle_{\partial\mathcal{T}_h},\nonumber\\
s_h^p(p_h,\widehat{p}_h;q_h,\widehat{q}_h)&=\langle h_F^{-1}(p_h-\widehat{p}_h,q_h-\widehat{q}_h) \rangle_{\partial\mathcal{T}_h}. \nonumber
\end{align}

To simplify notation, we define
\begin{align}
\bm{\sigma} &:=(\bm{r} ,\bm{u} ,\bm{u},{\nabla\times\bm{u}},p ,{p }),\\
\bm{\sigma}_h &:=(\bm{r}_h ,\bm{u}_h ,\widehat{\bm{u}}_h ,\widehat{\bm {c}}_h,p_h ,\widehat{p}_h ),\label{eq:sigmah}\\
\bm{\tau}_h &:=(\bm{s}_h ,\bm{v}_h ,\widehat{\bm{v}}_h,\widehat{\bm d}_h,q_h ,\widehat{q_h }),\label{eq:tauh}\\
\bm{\Sigma}_h &:=\bm{R}_h\times\bm{U}_h\times\widehat{\bm{U}}_h\times\widehat{\bm{C}}_h\times P_h\times \widehat{P_h},\\
\bm{\Sigma}_h^{\bm{g}} &:=\bm{R}_h\times\bm{U}_h\times\widehat{\bm{U}}^{\bm{g}_1}_h\times\widehat{\bm{C}}^{\bm{g}_2}_h\times P_h\times \widehat{P}^0_h,\\
\bm{\Sigma}_h^{\bm{0}} &:=\bm{R}_h\times\bm{U}_h\times\widehat{\bm{U}}^{\bm{0}}_h\times\widehat{\bm{C}}^{\bm{0}}_h\times P_h\times \widehat{P}^0_h,
\end{align}
and
\begin{align}
B_h(\bm{\sigma}_h,\bm{\tau}_h)&:=a_h(\bm{r}_h,\bm{s}_h)+b_h(\bm{u}_h,\widehat{\bm{u}}_h,\widehat{\bm c}_h;\bm{s}_h)\nonumber\\
&\,\,\,\quad+b_h(\bm{v}_h,\widehat{\bm{v}}_h,\widehat{\bm d}_h;\bm{r}_h)+c_h(p_h,\widehat{p}_h;\bm{v}_h)-s_h^u(\bm{u}_h,\widehat{\bm{u}}_h,\widehat{\bm c}_h; \bm{v}_h,\widehat{\bm{v}}_h,\widehat{\bm d}_h )\label{B_h}\\
&\qquad+c_h(q_h,\widehat{q}_h;\bm{u}_h)
+s_h^p(p_h,\widehat p_h;q_h,\widehat q_h),\nonumber\\
F_h(\bm{\tau}_h)&:=-(\bm{f},\bm{v}_h)+(g,q_h).
\end{align}
Then the HDG scheme \eqref{fhm1}--\eqref{fhm3} can be rewritten as:
\par\smallskip\noindent
Find $\bm{\sigma}_h\in \bm{\Sigma}^{\bm{g}}_h$ such that
\begin{align}\label{Bh_HDG}
B_h(\bm{\sigma}_h,\bm{\tau}_h)=F_h(\bm{\tau}_h)\qquad\forall\, \bm{\tau}_h\in\bm{\Sigma}^{\bm{0}}_h.
\end{align}

\subsection{Stability analysis}
We define the following semi-norms on spaces $\bm{U}_h\times \widehat{\bm{U}}_h\times\widehat{\bm{C}}_h$ and $P_h\times \widehat{P}_h$:
\begin{align}
\|(\bm{v},\widehat{\bm{v}},\widehat{\bm d})\|_{U}&:=\big( \|(\bm{v},\widehat{\bm{v}},\widehat{\bm d})\|^2_{\text{{\rm curl}}}+\|\bm{v}\|^2_{\text{{\rm div}}}\big)^{\frac{1}{2}},\\
\|(q,\widehat{q})\|^2_{P}&:=\big(\| h_T\nabla  q\|_{_{\mathcal T_h}}+\| h_F^{1/2}(q-\widehat{q}_h)\|^2_{\partial\mathcal{T}_h}\big)^{\frac{1}{2}}\label{def-p},
\end{align}
where
\begin{align*}
\|(\bm{v},\widehat{\bm{v}},\bm{d})\|^2_{\text{{\rm curl}}}&:=\|\nabla\times\nabla \times\bm{v}\|_{\mathcal T_h}^2+\|h_F^{-3/2}(\bm{n}\times(\bm{v}-\widehat{\bm{v}}))\|^2_{\partial\mathcal{T}_h}\nonumber\\
&\qquad+\|h_F^{-1/2}(\bm{n}\times(\nabla \times\bm{v}-\widehat{\bm d}))\|^2_{\partial\mathcal{T}_h},\nonumber\\
\|\bm{v}\|^2_{\text{{\rm div}}}&:=\|h_T\nabla_{h}\cdot\bm{v}\|_{\mathcal T_h}^2+\|h_F^{1/2}[\![\bm{n}\cdot\bm{v}]\!]\|^2_{\mathcal{F}^o_h}.
\end{align*}
Then we define the semi-norms on $\bm{\Sigma}_h$ as:
\begin{align}
\|\bm{\sigma} \|_{\bm{\Sigma}_h}&:=\big(\|\bm{r} \|_{_{\mathcal T_h}}^2+\|(\bm{u} ,\widehat{\bm{u} },\widehat{\bm c})\|^2_{U}+\|(p ,\widehat{p })\|^2_{P}\big)^{\frac{1}{2}},\label{norm-lambda}\\
\|\bm{\tau} \|_{\bm{\Sigma}_h}&:=\big(\|\bm{s} \|_{_{\mathcal T_h}}^2+\|(\bm{v} ,\widehat{\bm{v} },\widehat{\bm d})\|^2_{U}+\|(q ,\widehat{q })\|^2_{P}\big)^{\frac{1}{2}}.\label{norm-gamma}
\end{align}

\begin{lemma} The semi-norm	$\|\cdot\|_{U}$	defines a norm on  the space $\bm{U}_h\times\widehat{\bm{U}}^{\bm{0}}_h\times\widehat{\bm{C}}^{\bm{0}}_h$.
\end{lemma}
\begin{proof}
	Let $(\bm{v}_h,\widehat{\bm{v}}_h,\widehat{\bm d}_h)\in \bm{U}_h\times\widehat{\bm{U}}_h\times\widehat{\bm G}_h$. It suffices to show that
	$\|(\bm{v}_h,\widehat{\bm{v}}_h,\widehat{\bm d}_h)\|_{U}=0$ leads to $(\bm{v}_h,\widehat{\bm{v}}_h,\widehat{\bm d}_h)=(\bm{0},\bm{0},\bm 0)$, which can be checked easily. Actually, note that
	$\bm{n}\times\widehat{\bm{v}}_h=\bm{0}$ and $\bm{n}\cdot\widehat{\bm{v}}_h=0$ lead to $\widehat{\bm{v}}_h=\bm{0}$; Similarly,
	$\bm{n}\times\widehat{\bm{d}}_h=\bm{0}$ and $\bm{n}\cdot\widehat{\bm{d}}_h=0$ lead to $\widehat{\bm{d}}_h=\bm{0}$.	
\end{proof}

\begin{lemma}
	The semi-norm $\|(\cdot,\cdot)\|_{P}$ defines a norm on ${P}_h\times\widehat{P}^0_h$. Moreover, for all $(q_h,\widehat{q}_h)\in {P}_h\times\widehat{{P}}^0_h$, there holds
	\begin{align}
	\|(q_h,\widehat{q}_h)\|^2_{P}\sim \| h_T \nabla\mathcal{I}_kq_h\|_{\mathcal T_h}^2 + \| h_F^{-1/2}(q_h-\widehat{q}_h)\|^2_{\partial\mathcal{T}_h}. \label{relation}
	\end{align}
\end{lemma}
\begin{proof}
	By combining the definition of $\|(\cdot,\cdot)\|_{P}$ in \eqref{def-p}, the estimate \eqref{ick} and the triangle inequality, we have
	\begin{align}
	\|(q_h,\widehat{q}_h)\|^2_{P}&=\|h_T \nabla q_h\|_{\mathcal T_h}^2+\|h_F^{-1/2}(q_h-\widehat{q}_h)\|^2_{\partial\mathcal{T}_h} \nonumber\\
	&\lesssim \|h_T (\nabla q_h-\nabla \mathcal{I}^c q_h)\|_{\mathcal T_h}^2+\|h_T \nabla \mathcal{I}^c q_h\|_{\mathcal T_h}^2 +\|h_F^{-1/2}(q_h-\widehat{q}_h)\|^2_{\partial\mathcal{T}_h} \nonumber\\
	&\lesssim \|h_F^{-1/2}[\![q_h]\!]\|^2_{0,\mathcal{F}_h}+\|h_T \nabla\mathcal{I}^c q_h\|_{\mathcal T_h}^2+\|h_F^{-1/2}(q_h-\widehat{q}_h)\|^2_{\partial\mathcal{T}_h} \nonumber\\
	&\lesssim \|h_T \nabla\mathcal{I}^c q_h\|_{\mathcal T_h}^2+\|h_F^{-1/2}(q_h-\widehat{q}_h)\|^2_{\partial\mathcal{T}_h} .
	\end{align}
	On the other hand,
	\begin{align}
	\|h_T \nabla\mathcal{I}^c q_h\|_{\mathcal T_h}^2 &\lesssim \|h_T (\nabla\mathcal{I}^c q_h-\nabla q_h)\|_{\mathcal T_h}^2+\|\nabla q_h\|_{\mathcal T_h}^2\nonumber\\
	&\lesssim \|h_F^{-1/2}[\![q_h]\!]\|^2_{0,\mathcal{F}_h}+\|h_T \nabla q_h\|_{\mathcal T_h}^2 \nonumber\\
	&\lesssim \|h_F^{-1/2}(q_h-\widehat{q}_h)\|^2_{\partial\mathcal{T}_h}+\|\nabla q_h\|_{\mathcal T_h}^2\nonumber\\
	&= \|(q_h,\widehat{q}_h)\|^2_{P}.
	\end{align}
	Therefore, the estimate \eqref{relation} holds.
	\par
	Next, we prove  $\|(\cdot,\cdot)\|_P$ is a norm on ${P}_h\times\widehat{P}^0_h$.  For any $(q_h, \widehat{q}_h)\in {P}_h\times\widehat{P}^0_h$ such that
	$\|h_T \nabla q_h\|_{\mathcal T_h}^2 +\|h_F^{-1/2}(q_h-\widehat{q}_h)\|^2_{\partial\mathcal{T}_h}=0$,
	we know that $q_h$ is piecewise constants, and $q_h=\widehat{q}_h$ on every face. Moreover, $q_h=\widehat{q}_h=0$ on boundary faces. Therefore, $q_h=\widehat{q}_h=0$. This completes the proof.
\end{proof}

\begin{theorem}[Discrete inf-sup condition]\label{Th45} The following stability estimates hold true for $B_h$.
	\begin{align}
	\sup_{\bm{\tau}_h\in\bm{\Sigma}^{\bm{0}}_h,\bm{\tau}_h\neq\bm{0}}\frac{B_h(\bm{\sigma}_h,\bm{\tau}_h)}{\|\bm{\tau}_h\|_{\bm{\Sigma}_h}}
	&\gtrsim  \|\bm{\sigma}_h\|_{\bm{\Sigma}_h},\label{lbb1}\\
	\sup_{\bm{\sigma}_h\in\bm{\Sigma}^{\bm{0}}_h,\bm{\sigma}_h\neq\bm{0}}\frac{B_h(\bm{\sigma}_h,\bm{\tau}_h)}{\|\bm{\sigma}_h\|_{\bm{\Sigma}_h}}
	&\gtrsim  \|\bm{\tau}_h\|_{\bm{\Sigma}_h}. \label{lbb2}
	\end{align}
\end{theorem}

\begin{proof}
	We use the following five steps to derive \eqref{lbb1}--\eqref{lbb2}.

	\textbf{Step one:}
	
	Taking $\bm{\tau}^1_h=(\bm{r}_h ,-\bm{u}_h ,-\widehat{\bm{u}}_h ,-\widehat{\bm {c}}_h,p_h ,\widehat{p}_h )\in\bm{\Sigma}^{\bm{0}}_h$, then by the definitions of $\bm{\sigma}_h$,  $B_h$ and the norm $\|\cdot\|_{\bm{\bm{\Sigma}}_h}$ (cf. \eqref{eq:sigmah}, \eqref{B_h} and \eqref{norm-lambda}--\eqref{norm-gamma}), we have
	\begin{align}\label{eq:tauh1}
	\|\bm{\tau}_h^1\|_{\bm{\Sigma}_h}=\|\bm{\sigma}_h\|_{\bm{\Sigma}_h},
	\end{align}
	and
	\begin{align}\label{eq:Bh1}
	\begin{split}
	B_h(\bm{\sigma}_h,\bm{\tau}^1_h)&=\|\bm{r}_h\|_{\mathcal T_h}^2+\|h_F^{-3/2}\bm{n}\times(\bm{u}_h-\widehat{\bm{u}}_h)\|^2_{\partial\mathcal{T}_h}\\
	&\quad+\|h_F^{-1/2}(\bm{n}\times(\nabla \times\bm{u}_h-\widehat{\bm {c}}_h))\|^2_{\partial\mathcal{T}_h}+\|h_F^{-1/2}(p_h-\widehat{p}_h)\|^2_{\partial\mathcal{T}_h}.
	\end{split}
	\end{align}

	\textbf{Step two:}
	
	By taking $\bm{\tau}^2_h=(-\nabla \times\nabla \times\bm{u}_h,\bm{0},\bm{0},\bm{0},0,0) \in\bm{\Sigma}^{\bm{0}}_h$ we have
	\begin{align}\label{eq:tauh2}
	\|\bm{\tau}_h^2\|_{\bm{\Sigma}_h}=\|\nabla \times\nabla \times\bm{u}_h\|_{_{\mathcal T_h}}\le\|\bm{\sigma}_h\|_{\bm{\Sigma}_h}.
	\end{align}
	By the definition of $B_h$, integration by parts and inverse inequality, we have
	\begin{align}\label{eq:Bh2}
	\begin{split}
	B_h(\bm{\sigma}_h,\bm{\tau}^2_h)&=-(\bm{r}_h,\nabla \times\nabla \times\bm{u}_h)+\|\nabla \times\nabla \times\bm{u}_h\|_{\mathcal T_h}^2\\
	&\quad\,\,\,+\langle\nabla \times\nabla \times\nabla \times\bm{u}_h,\bm{n}\times(\bm{u}_h- \widehat{\bm{u}}_h  )  \rangle_{\partial\mathcal{T}_h}\\
	&\qquad+\langle\nabla \times\nabla \times\bm{u}_h,\bm{n}\times(\nabla \times\bm{u}_h- \widehat{\bm {c}}_h  )  \rangle_{\partial\mathcal{T}_h}\\
	&\ge \frac{1}{2}\|\nabla \times\nabla \times\bm{u}_h\|_{\mathcal T_h}^2-C_1\|\bm{r}_h\|_{\mathcal T_h}^2\\
	&\quad-C_1 \|h_F^{-3/2}\bm{n}\times(\bm{u}_h-\widehat{\bm{u}}_h)\|^2_{\partial\mathcal{T}_h}\\
	&\quad-C_1 \|h_F^{-1/2}\bm{n}\times(\nabla \times\bm{u}_h-\widehat{\bm {c}}_h)\|^2_{\partial\mathcal{T}_h}.
	\end{split}
	\end{align}

	\textbf{Step three:}
	
	Let $r_h=h_T^2\nabla \cdot\bm{u}_h$, $\widehat{r}_h=h_F[\![\bm{n}\cdot \bm{u}_h]\!]$ on $\mathcal{F}_h^o$ and $\widehat{r}_h=0$ on $\Gamma$. Taking $\bm{\tau}^3_h=(\bm{0},\bm{0},\bm{0},\bm{0},r_h,\widehat{r}_h)\in\bm{\Sigma}^{\bm{0}}_h$, then by the definition of $\|\cdot\|_{\bm{\bm{\Sigma}}_h}$ and inverse inequality we have
	\begin{align}\label{eq:tauh3}
	\begin{split}
	\|\bm{\tau}^3_h\|_{\bm{\Sigma}_h}^2	&= \|\nabla  r_h\|^2_{0}+ \|h_F^{-1/2}(r_h-\widehat{r}_h)\|^2_{\partial\mathcal{T}_h}\\
	&\lesssim \| h_T\nabla \cdot\bm{u}_h\|_{\mathcal T_h}^2+\|h_F^{1/2}[\![\bm{n}\cdot\bm{u}_h]\!]\|^2_{0,\mathcal{F}_h^o}\\
	&\le \|\bm{\sigma}_h\|_{\bm{\Sigma}_h}^2.
	\end{split}
	\end{align}
	Moreover,
	\begin{align}\label{eq:Bh3}
	\begin{split}
	B_h(\bm{\sigma}_h,\bm{\tau}^3_h)&=\|h_T\nabla\cdot\bm{u}_h\|_{\mathcal T_h}^2+\| h_F^{1/2}[\![\bm{n}\cdot\bm{u}_h]\!]\|^2_{0,\mathcal{F}_h^o} \\
	&\qquad+\langle h_F^{-1}(p_h-\widehat{p}_h),r_h-\widehat{r}_h \rangle_{\partial\mathcal{T}_h}\\
	&\ge\frac{1}{2}\|\bm{u}_h\|^2_{{\rm div}}-C_2\|h_F^{-1/2}(p_h-\widehat{p}_h)\|^2_{\partial\mathcal{T}_h}.
	\end{split}
	\end{align}

	\textbf{Step four:}
	
	We then take $\bm{\tau}^4_h=(\bm{0},-\nabla\mathcal{I}_k p_h,-\bm n\times\nabla \mathcal{I}_kp_h\times\bm n,\bm{0},0,0)\in\bm{\Sigma}^{\bm{0}}_h$. Similar to previous steps, we have
	\begin{align}\label{eq:tauh4}
	\begin{split}
	\|\bm{\tau}^4_h\|_{\bm{\Sigma}_h}^2&= \|h_T\nabla \cdot\nabla  \mathcal{I}_kp_h\|_{\mathcal T_h}^2+ \|h_F^{1/2}[\![\nabla  \mathcal{I}_kp_h\cdot\bm{n}]\!]\|_{0,\mathcal{F}_h^o}\\
	&\lesssim \|\nabla  \mathcal{I}_kp_h\|_{\mathcal T_h}^2\\
	&\lesssim \|\bm{\sigma}_h\|_{\bm{\Sigma}_h}^2.
	\end{split}
	\end{align}
	Moreover,
	\begin{align}\label{eq:Bh4}
	\begin{split}
	B_h(\bm{\sigma}_h,\bm{\tau}^4_h)&=(\nabla\mathcal{I}_k p_h,\nabla \times\nabla \times\bm{r}_h)_{\mathcal T_h}+\langle\bm n\times(\bm n\times\nabla \mathcal I_kp_h\times\bm n ),\nabla \times \bm r_h \rangle_{\partial\mathcal{T}_h}\\
	&\quad-(\nabla \cdot \nabla \mathcal{I}_k p_h, p_h)_{\mathcal T_h}+ \langle \bm{n}\cdot \nabla\mathcal{I}_k p_h,\widehat{p}_h \rangle_{\partial\mathcal{T}_h}\\
	&=\|\nabla\mathcal{I}_kp_h\|^2_{_{\mathcal T_h}}+ (\nabla  p_h- \nabla \mathcal{I}_kp_h,\nabla  p_h)_{\mathcal T_h}+  \langle  \bm{n}\cdot\nabla\mathcal{I}_kp_h,\widehat{p}_h -p_h \rangle_{\partial\mathcal{T}_h}\\
	&\ge \frac{1}{2}\| \nabla  p_h\|_{\mathcal T_h}^2-C_3  \| h_F^{-1/2}(p_h-\widehat{p}_h)\|^2_{\partial\mathcal{T}_h}-C_3\|\bm{r}_h\|_{\mathcal T_h}^2\\
	&\quad-C_3\| h_F^{-1/2}\bm{n}\times(\bm{v}_h-\widehat{\bm{v}}_h) \|^2_{\partial\mathcal{T}_h}.
	\end{split}
	\end{align}

	\textbf{Step five:}
	
	Take $	C_0=\max(C_1,C_2)+C_3+1$ and $\bm{\tau}_h=C_0\bm{\tau}_h^1+\bm{\tau}_h^2+\bm{\tau}_h^3+\bm{\tau}_h^4$,
	then it follows from \eqref{eq:tauh1}--\eqref{eq:Bh4} that
	\begin{align*}
	\|\bm{\tau}_h\|_{\bm{\Sigma}_h}&\lesssim \|\bm{\sigma}_h\|_{\bm{\Sigma}_h},\\
	B_h(\bm{\sigma}_h,\bm{\tau}_h)&\gtrsim	\|\bm{r}_h\|_{_{\mathcal T_h}}^2+\|(\bm{u}_h,\widehat{\bm{u}}_h,\widehat{\bm c}_h)\|^2_{U}+\|(p_h,\widehat{p}_h)\|^2_{P}=\|\bm{\sigma}_h\|^2_{\bm{\Sigma}_h}.
	\end{align*}
	Hence
	\begin{align*}
	B_h(\bm{\sigma}_h,\bm{\tau}_h)\gtrsim \|\bm{\sigma}_h\|_{\bm{\Sigma}_h}\|\bm{\tau}_h\|_{\bm{\Sigma}_h},
	\end{align*}
	which implies \eqref{lbb1}. Since $B_h$ is symmetric, \eqref{lbb2} also holds.
\end{proof}

The following corollary is a direct consequence of Theorem~\ref{Th45}.
\begin{corollary}
	The HDG scheme \eqref{Bh_HDG} admits a unique solution $\bm{\sigma}_h\in\bm{\Sigma}_h^{\bm{g}}$.
\end{corollary}

\section{Error estimates}

\subsection{Primary estimates}
\begin{lemma} \label{Jh}Let $(\bm{r},\bm{u},p)$ be the solution of \eqref{mix0source}, $\bm{\sigma}$ and $\bm{\mathcal{J}}_h\bm{\sigma}$ be defined by
	\begin{align*}
	\bm{\sigma}&:=(\bm{r},\bm{u},{\bm{u}},\nabla\times\bm u,p,p),\\
	\bm{\mathcal{J}}_h\bm{\sigma}&:=(\bm{\Pi}_{k-1}^{o}\bm{r},\bm{\Pi}^{{\rm curl}}_{h,k}\bm{u},\bm n\times\bm{\Pi}^{{\rm curl}}_{h,k}\bm{u}\times\bm n,\bm{\Pi}^{\partial}_{k-1}\nabla\times\bm{u},\Pi^o_{k}p,\Pi^{\partial}_kp).
	\end{align*}
	Then we have
	\begin{align}
	B_h(\bm{\mathcal{J}}_h\bm{\sigma},\bm{\tau}_h)=F_h(\bm{\tau}_h)+E^{\mathcal{J}}_h(\bm{\sigma};\bm{\tau}_h)\quad\forall\, \bm{\tau}_h\in\bm{\Sigma}^{\bm{0}}_h, \label{510}
	\end{align}
	where
	\begin{align}\label{eq:EJh}
	\begin{split}
	E^{\mathcal{J}}_h(\bm{\sigma};\bm{\tau}_h)&= -\langle \bm{n}\times(\widehat{\bm{v}}_h-\bm{v}_h),\nabla \times(\bm{\Pi}_{k-1}^{o}\bm{r}-\bm{r}) \rangle_{\partial\mathcal{T}_h} \\
	&\quad\,\,-\langle \bm{n}\times(\widehat{\bm d}_h-\nabla \times\bm{v}_h),\bm{\Pi}_{k-1}^{o}\bm{r}-\bm{r} \rangle_{\partial\mathcal{T}_h} \\
	&\quad\,\,\,\,-\langle h_F^{-1} \bm{n}\times(  \nabla\times\bm{\Pi}^{{\rm curl}}_{h,k}\bm{u}-\nabla\times\bm{u}        ), \bm{n}\times(\nabla \times\bm{v}_h-\widehat{\bm d}_h )  \rangle_{\partial\mathcal{T}_h}\\
	&\quad\,\,\,\,\,\,-\langle \widehat{q}_h-q_h,(\bm{\Pi}^{{\rm curl}}_{h,k}\bm{u}-\bm u)\cdot\bm n \rangle_{\partial\mathcal{T}_h}\\
	&\qquad\,\,+\langle h_F^{-1}( \Pi_{k}^op-p),q_h-\widehat{q}_h \rangle_{\partial\mathcal{T}_h}-(\nabla  q_h,\bm{\Pi}^{{\rm curl}}_{h,k}\bm{u}-\bm u).
	\end{split}
	\end{align}
\end{lemma}
\begin{proof} By the definitions of $a_h$, $b_h$, integration by parts and the fact that $m\le k-1$, we arrive at
	\begin{align*}
	&a_h(\bm{\Pi}_{k-1}^{o}\bm{r},\bm{s}_h)+b_h(\bm{\Pi}^{{\rm curl}}_{h,k}\bm{u},\bm n\times\bm{\Pi}^{{\rm curl}}_{h,k}\bm{u}\times\bm n;\bm{s}_h)\nonumber\\
	&\quad\quad= (\bm{\Pi}_{k-1}^{o}\bm{r},\bm{s}_h)-(\bm{\Pi}^{{\rm curl}}_{h,k}\bm{u},\nabla \times\nabla \times\bm{s}_h)\nonumber\\
	&\qquad\quad\quad-\langle \bm{n}\times\bm n\times\bm{\Pi}^{{\rm curl}}_{h,k}\bm{u}\times\bm n,\nabla \times\bm{s}_h \rangle_{\partial\mathcal{T}_h}
	-\langle \bm{n}\times\bm{\Pi}^{\partial}_{k-1}\nabla\times\bm{u},\bm{s}_h \rangle_{\partial\mathcal{T}_h}  \nonumber\\
	&\quad\quad= (\bm{r},\bm{s}_h)-(\nabla\times\bm{\Pi}^{{\rm curl}}_{h,k}\bm{u},\nabla \times\bm{s}_h)
	-\langle \bm{n}\times\nabla\times\bm{u},\bm{s}_h \rangle_{\partial\mathcal{T}_h}\nonumber\\
	&\quad\quad= (\bm{r},\bm{s}_h)+(\nabla\times(\bm u-\bm{\Pi}^{{\rm curl}}_{h,k}\bm{u}),\nabla \times\bm{s}_h)\nonumber\\
	&\qquad\quad\quad-(\nabla\times\bm u,\nabla \times\bm{s}_h)-\langle \bm{n}\times\nabla\times\bm{u},\bm{s}_h \rangle_{\partial\mathcal{T}_h}.
	\end{align*}
	Using integration by parts and the fact that $\bm r=\nabla\times\nabla\times\bm u$, and $(\nabla\times(\bm u-\bm{\Pi}^{{\rm curl}}_{h,k}\bm{u}),\nabla \times\bm{s}_h)=0$, we get
	\begin{align}\label{eq:ahest}
	a_h(\bm{\Pi}_{k-1}^{o}\bm{r},\bm{s}_h)+b_h(\bm{\Pi}^{{\rm curl}}_{h,k}\bm{u},\bm n\times\bm{\Pi}^{{\rm curl}}_{h,k}\bm{u}\times\bm n;\bm{s}_h)= 0.
	\end{align}
	\vspace{0.02in}
	\par
	By the definitions of $b_h$, $c_h$, $s_h^u$ and integration by parts, one can get
	\begin{align*}
	&b_h(\bm{v}_h,\widehat{\bm{v}}_h;\bm{\Pi}_{k-1}^{o}\bm{r})+c_h(\Pi^o_{k}p,\Pi^{\partial}_kp;\bm{v}_h)
	-s^u_h(  \bm{\Pi}^{{\rm curl}}_{h,k}\bm{u},\bm n\times\bm{\Pi}^{{\rm curl}}_{h,k}\bm{u}\times\bm n; \bm{v}_h,\widehat{\bm{v}}_h )_{\mathcal T_h} \nonumber\\
	&\quad\quad=-(\bm{v}_h,\nabla \times\nabla \times\bm{\Pi}_{k-1}^{o}\bm{r})_{\mathcal T_h}-\langle \bm{n}\times\widehat{\bm{v}}_h,\nabla \times\bm{\Pi}_{k-1}^{o}\bm{r} \rangle_{\partial\mathcal{T}_h}  \nonumber\\
	&\,\,\,\,\quad\quad\quad\quad-\langle \bm{n}\times\widehat{\bm d}_h,\bm{\Pi}_{k-1}^{o}\bm{r} \rangle_{\partial\mathcal{T}_h}+(\nabla \cdot\bm{v}_h, \Pi^o_{k}p)_{\mathcal T_h}-\langle \bm{n}\cdot\bm{v}_h,\Pi^{\partial}_kp \rangle_{\partial\mathcal{T}_h} \nonumber\\
	&\,\,\quad\qquad\quad\quad-\langle h_F^{-1} \bm{n}\times(\nabla\times\bm{\Pi}^{{\rm curl}}_{h,k}\bm{u}-\Pi^{\partial}_{s}\nabla\times\bm{u}), \bm{n}\times(\nabla \times\bm{v}_h-\widehat{\bm d}_h )  \rangle_{\partial\mathcal{T}_h}\nonumber\\
	&\quad\quad=-(\nabla \times\nabla \times\bm{v}_h,\bm{\Pi}_{k-1}^{o}\bm{r})_{\mathcal T_h}-(\bm{v}_h,\nabla p)_{\mathcal T_h} \nonumber\\
	&\quad\quad\quad\quad-\langle \bm{n}\times(\widehat{\bm{v}}_h-\bm v_h),\nabla \times\bm{\Pi}_{k-1}^{o}\bm{r} \rangle_{\partial\mathcal{T}_h} -\langle \bm{n}\times(\widehat{\bm d}_h-\nabla \times\bm{v}_h),\bm{\Pi}_{k-1}^{o}\bm{r} \rangle_{\partial\mathcal{T}_h} \nonumber\\
	&\,\,\quad\qquad\quad\quad-\langle h_F^{-1} \bm{n}\times(  \nabla\times\bm{\Pi}^{{\rm curl}}_{h,k}\bm{u}-\nabla\times\bm{u}), \bm{n}\times(\nabla \times\bm{v}_h-\widehat{\bm d}_h ) \rangle_{\partial\mathcal{T}_h}.
	\end{align*}
	Since $\langle \bm{n}\times\widehat{\bm{v}}_h,\nabla \times\bm{r} \rangle_{\partial\mathcal{T}_h}=0$ and
	$\langle \bm{n}\times\widehat{\bm d}_h,\bm{r} \rangle_{\partial\mathcal{T}_h}=0$, it then follows from \eqref{eq:bcsh} and integration by parts that
	\begin{align}
	\begin{split}
	&b_h(\bm{v}_h,\widehat{\bm{v}}_h;\bm{\Pi}_{k-1}^{o}\bm{r})+c_h(\Pi^o_{k}p,\Pi^{\partial}_kp;\bm{v}_h)
	-s^u_h(  \bm{\Pi}^{{\rm curl}}_{h,k}\bm{u},\bm n\times\bm{\Pi}^{{\rm curl}}_{h,k}\bm{u}\times\bm n; \bm{v}_h,\widehat{\bm{v}}_h ) \\
	&\quad\quad=-(\bm{v}_h,\nabla\times\nabla\times\bm r+\nabla p)_{\mathcal T_h} -\langle \bm{n}\times(\widehat{\bm{v}}_h-\bm{v}_h),\nabla \times(\bm{\Pi}_{k-1}^{o}\bm{r}-\bm{r}) \rangle_{\partial\mathcal{T}_h} \\
	& \quad\quad\quad\quad-\langle \bm{n}\times(\widehat{\bm d}_h-\nabla \times\bm{v}_h),\bm{\Pi}_{k-1}^{o}\bm{r}-\bm{r} \rangle_{\partial\mathcal{T}_h} \\
	&\qquad\quad\quad\quad-\langle h_F^{-1} \bm{n}\times(  \nabla\times\bm{\Pi}^{{\rm curl}}_{h,k}\bm{u}-\nabla\times\bm{u}        ), \bm{n}\times(\nabla \times\bm{v}_h-\widehat{\bm d}_h )  \rangle_{\partial\mathcal{T}_h}\\
	&\quad\quad=-(\bm{v}_h,\bm f) -\langle \bm{n}\times(\widehat{\bm{v}}_h-\bm{v}_h),\nabla \times(\bm{\Pi}_{k-1}^{o}\bm{r}-\bm{r}) \rangle_{\partial\mathcal{T}_h} \\
	&\quad\quad\quad\quad-\langle \bm{n}\times(\widehat{\bm d}_h-\nabla \times\bm{v}_h),\bm{\Pi}_{k-1}^{o}\bm{r}-\bm{r} \rangle_{\partial\mathcal{T}_h} \\
	&\qquad\quad\quad\quad-\langle h_F^{-1} \bm{n}\times(  \nabla\times\bm{\Pi}^{{\rm curl}}_{h,k}\bm{u}-\nabla\times\bm{u}        ), \bm{n}\times(\nabla \times\bm{v}_h-\widehat{\bm d}_h )  \rangle_{\partial\mathcal{T}_h}.
	\end{split}
	\end{align}
	\vspace{0.01in}
	By the definition of $c_h$ and integration by parts we have
	\begin{align}
	c_h(q_h,\widehat{q}_h;\bm{\Pi}^{{\rm curl}}_{h,k}\bm{u}) &= (q_h,\nabla \cdot\bm{\Pi}^{{\rm curl}}_{h,k}\bm{u})-\langle\widehat{q}_h,\bm{\Pi}^{{\rm curl}}_{h,k}\bm{u}\cdot\bm n \rangle_{\partial\mathcal{T}_h}\nonumber\\
	&=-(\nabla  q_h,\bm{\Pi}^{{\rm curl}}_{h,k}\bm{u})-\langle \widehat{q}_h-q_h,\bm{\Pi}^{{\rm curl}}_{h,k}\bm{u}\cdot\bm n \rangle_{\partial\mathcal{T}_h} \nonumber\\
	&=-\langle \widehat{q}_h-q_h,(\bm{\Pi}^{{\rm curl}}_{h,k}\bm{u}-\bm u)\cdot\bm n \rangle_{\partial\mathcal{T}_h}- (\nabla  q_h,\bm{\Pi}^{{\rm curl}}_{h,k}\bm{u}-\bm u)_{\mathcal T_h}+(g,q_h)_{\mathcal T_h},
	\end{align}
	\vspace{0.01in}
	where we have used the fact $\langle
	\widehat q_h,\bm u\cdot\bm n \rangle_{\partial\mathcal{T}_h}=0$ and $\nabla\cdot\bm u=g$. By the definition of $s^p_h$ we get
	\begin{align}\label{eq:shest}
	s_h^p(\Pi_{k}^op,\Pi_k^{\partial}p;q_h,\widehat{q}_h)=\langle h_F^{-1}( \Pi_{k}^op-p),q_h-\widehat{q}_h \rangle_{\partial\mathcal{T}_h}.
	\end{align}

	Finally the desired result \eqref{510} follows from the definition \eqref{B_h} and \eqref{eq:ahest}--\eqref{eq:shest}.
	
\end{proof}

%
We recall the result in \cite{MR3508837}. For any $(\bm v,\widehat{\bm v})\in [L^2(\Omega)]^3\times[L^2(\partial\mathcal T_h)]^3$, and for any $T\in\mathcal T_h$, there exists an interpolation $\bm{\mathcal I}_T(\bm v,\widehat{\bm v})\in [\mathcal P_{k+3}(T)]^3$ such that
\begin{subequations}
	\begin{align}
	(\bm{\mathcal I}_T(\bm v,\widehat{\bm{v}}),\bm w_h)_T&=(\bm v,\bm w_h)_T,\label{i1}\\
	\langle \bm{\mathcal I}_T(\bm v,\widehat{\bm{v}}), \widehat{\bm{w}}_h \rangle_{F}&=\langle  \widehat{\bm{v}},\widehat{\bm{w}}_h \rangle_{F},\label{i2}
	\end{align}
\end{subequations}
for all  $(\bm{w}_h,\widehat{\bm{w}}_h)\in
[\mathcal P_{k}(T)\color{black} ]^3  \color{black}\times [\mathcal P_k(F)\color{black} ]^3  \color{black}$, and $F\color{black}\subset\color{black}\partial T$. We define $\bm{\mathcal I}_h|_T=\bm{\mathcal I}_T$, if $\bm v_h|_T\in [\mathcal P_k(T)]^3$, 
$\widehat{\bm v}_h\in[\mathcal P_k(F)]^3$ for all $F\subset \partial T$,
it holds
\begin{subequations}
	\begin{align}
	\|\bm v_h- \bm{\mathcal I}_h(\bm v_h,\widehat{\bm{v}}_h)\|_{\mathcal T_h}  &\lesssim   \| h_T^{1/2}(\bm v_h-\widehat{\bm v}_h)\|_{\partial \mathcal T_h},\label{i3}\\
	\|\nabla(\bm v_h-\bm{\mathcal I}_h(\bm v_h,\widehat{\bm{v}}_h))\|_{\mathcal T_h} &\lesssim \| h_T^{-1/2}(\bm v_h-\widehat{\bm v}_h)\|_{\partial \mathcal T_h} . 	\label{i4}
	\end{align}
\end{subequations}
In addition, if $\widehat{\bm v}_h\in [\mathcal P_k(\mathcal F_h)]^3$ and $\widehat{\bm v}_h|_{\partial\Omega}=\bm 0$, then $\bm{\mathcal I}_h(\bm v_h,\widehat{\bm v}_h)\in [H^1_0(\Omega)]^3$.

We define 
\begin{align}
\bm{\Pi}_T(\bm v,\widehat{\bm v}):=\bm{\mathcal I}_T(\bm v,\widehat{\bm 
	v}+(\bm n\cdot\bm v)\bm n).\label{def_pi}
\end{align}

\begin{lemma} \label{interpolation_projection}
	For any $T\in\mathcal{T}_h$
	and
	$(\bm{v},\widehat{\bm{v}})\in H^1(T)\times L^2(\partial T)$, we have
	\begin{subequations}\label{eq:interPi}
		\begin{align}
		(\bm{\Pi}_T(\bm v,\widehat{\bm{v}}),\bm w_h)_T&=(\bm v,\bm w_h)_T,\label{ia}\\
		\langle \bm n\times\bm{\Pi}_T(\bm v,\widehat{\bm{v}}),\bm n\times \widehat{\bm{w}}_h \rangle_{F}&=\langle \bm n\times \widehat{\bm{v}}, \bm n\times\widehat{\bm{w}}_h \rangle_{F},\label{ib}
		\end{align}
	\end{subequations}
	for all  $(\bm{w}_h,\widehat{\bm{w}}_h)\in
	[\mathcal P_{k}(T)\color{black} ]^3  \color{black}\times [\mathcal P_k(F)\color{black} ]^3  \color{black}$, and $F\color{black}\subset\color{black}\partial T$. And the following approximation properties hold true for 
	$(\bm v_h,\widehat {\bm{v}}_h)\in \bm V_h\times \widehat {\bm V}_h$
	\begin{align}
	\|\bm v_h- \bm{\Pi}_h(\bm v_h,\widehat{\bm{v}}_h)\|_{\mathcal T_h}  &\lesssim   \| h_T^{1/2}\bm n\times(\bm v_h-\widehat{\bm v}_h)\|_{\partial \mathcal T_h} \label{2_error},\\
	\|\nabla\times(\bm v_h-\bm{\Pi}_h(\bm v_h,\widehat{\bm{v}}_h))\|_{\mathcal T_h} &\lesssim \| h_T^{-1/2}\bm n\times(\bm v_h-\widehat{\bm v}_h)\|_{\partial \mathcal T_h}. \label{15_error}
	\end{align}
	Moreover, we define $\bm{\Pi}_h|_T=\bm{\Pi}_T$,
	then 
	$\bm{\Pi}_h(\bm v_h, \bm v_h)\in \bm H_0({\rm curl};\Omega)$ for all $(\bm v_h,\widehat {\bm v}_h)\in \bm V_h\times\widehat {\bm V}_h$.
	\color{black}
\end{lemma}
\begin{proof}
	By \eqref{i1} and \eqref{def_pi} we have
	\begin{align*}
	(\bm{\Pi}_T(\bm v,\widehat{\bm v}),\bm w_h)_T=(\bm{\mathcal I}_T(\bm v,\widehat{\bm 
		v}+(\bm n\cdot\bm v)\bm n),\bm w_h)_T=(\bm v,\bm w_h)_T.
	\end{align*}
	By \eqref{i2} and \eqref{def_pi}, it holds
	\begin{align*}
	\langle \bm n\times\bm{\Pi}_T(\bm v,\widehat{\bm{v}}),\bm n\times \widehat{\bm{w}}_h \rangle_{F}&=	\langle \bm n\times\bm{\mathcal I}_T(\bm v,\widehat{\bm{v}}),\bm n\times \widehat{\bm{w}}_h \rangle_{F} \nonumber\\
	&=\langle \bm{\mathcal I}_T(\bm v,\widehat{\bm{v}}),\bm n\times \widehat{\bm{w}}_h\times\bm n \rangle_{F} \nonumber\\
	&=\langle \widehat{\bm{v}},\bm n\times \widehat{\bm{w}}_h\times\bm n \rangle_{F} \nonumber\\
	&=\langle\bm n\times \widehat{\bm{v}},\bm n\times \widehat{\bm{w}}_h \rangle_{F}.
	\end{align*}
	We use \eqref{i3}, \eqref{def_pi} and the fact $\bm n\cdot\widehat{\bm v}_h=\bm 0$ to get
	\begin{align*}
	\|\bm v_h- \bm{\Pi}_h(\bm v_h,\widehat{\bm{v}}_h)\|_{\mathcal T_h}  &\lesssim   \| h_T^{1/2}(\bm v_h-(\widehat{\bm v}_h+(\bm n\cdot\bm v_h)\bm n))\|_{\partial \mathcal T_h}\nonumber\\
	&= \| h_T^{1/2}(\bm v_h-(\bm n\cdot\bm v_h)\bm n)-(\widehat{\bm v}_h-(\bm n\cdot\widehat{\bm v}_h)\bm n)\|_{\partial \mathcal T_h}\nonumber\\
	&=\| h_T^{1/2}\bm n\times(\bm v_h-\widehat{\bm  v}_h)\times\bm n\|_{\partial \mathcal T_h}\nonumber\\
	&\le\| h_T^{1/2}\bm n\times(\bm v_h-\widehat{\bm  v}_h)\|_{\partial \mathcal T_h}.
	\end{align*}
	\eqref{15_error} is followed by the proof similar to the above one, \eqref{i4}, and the fact $\|\nabla\times(\bm v_h-\bm{\Pi}_h(\bm v_h,\widehat{\bm v}_h))\|_{\mathcal T_h}\lesssim \|\nabla(\bm v_h-\bm{\Pi}_h(\bm v_h,\widehat{\bm v}_h))\|_{\mathcal T_h}$. We use \eqref{def_pi} to get
	$\bm n_F\times\bm{\Pi}_T(\bm v_h,\widehat{\bm v}_h)=\bm{\mathcal I}_T(\bm n_F\times\bm v,\bm n_F\times\widehat{\bm v})$ on every face $F\subset \partial T$. Since $\bm n\times\widehat{\bm v}\in [\mathcal P_k(\mathcal F_h)]^3$ and $\bm n\times\widehat{\bm v}|_{\partial\Omega}=\bm 0$, then $\bm{\mathcal I}_T(\bm n_F\times\bm v_h,\bm n_F\times\widehat{\bm v}_h)$ is continuous on $F$ and $\bm{\mathcal I}_T(\bm n_F\times\bm v_h,\bm n_F\times\widehat{\bm v}_h)|_{\partial\Omega}=\bm 0$, so $\bm{\Pi}_h(\bm v_h,\widehat{\bm v}_h)\in \bm H_0({\rm curl};\Omega)$.
\end{proof}
\color{black}

\begin{lemma} \label{lemma53}
	We have the following error estimates
	\begin{align*}
	E^{\mathcal{J}}_h(\bm{\sigma};\bm{\tau}_h)&\lesssim h^2\|\nabla\times\nabla\times \bm r\|_{_{\mathcal T_h}}\|h_F^{-3/2}\bm n\times(\bm v_h-\widehat{\bm v}_h)\|_{\partial\mathcal{T}_h}\\
	&\quad+h^s\|\bm r\|_{s}\left(\|h_F^{-1/2}\bm{n}\times(\widehat{\bm d}_h-\nabla \times\bm{v}_h)\|_{0,\partial \mathcal{T}_h}+\|h_F^{-3/2}\bm n\times(\bm v_h-\widehat{\bm v}_h)\|_{\partial\mathcal{T}_h}\right)\\
	&\,\,\,\,\quad+h^{s}\|\bm u\|_{s}\left(\|\nabla  q_h\|_{_{\mathcal T_h}}+\|h_F^{-1/2}(q_h-\widehat{q}_h)\|_{\partial\mathcal{T}_h}\right) \\
	&\qquad\,\,+h^{s}\|p\|_{s+1}\|h_F^{-1/2}(q_h-\widehat{q}_h)\|_{\partial\mathcal{T}_h}\\
	&\qquad\quad+h^{\min(k-1,s)}\|\nabla\times\bm{u}\|_{s+1} \|h_F^{-1/2}\bm{n}\times(\widehat{\bm d}_h-\nabla \times\bm{v}_h)\|_{0,\partial \mathcal{T}_h}.
	\end{align*}
\end{lemma}
\begin{proof}
	For simplicity we define
	\begin{align*}
	E_1=& -\langle \bm{n}\times(\widehat{\bm{v}}_h-\bm{v}_h),\nabla \times(\bm{\Pi}_{k-1}^{o}\bm{r}-\bm{r}) \rangle_{\partial\mathcal{T}_h}, \nonumber\\
	E_2=&-\langle \bm{n}\times(\widehat{\bm d}_h-\nabla \times\bm{v}_h),\bm{\Pi}_{k-1}^{o}\bm{r}-\bm{r} \rangle_{\partial\mathcal{T}_h}, \nonumber\\
	E_3=&-\langle h_F^{-1} \bm{n}\times(  \nabla\times\bm{\Pi}^{{\rm curl}}_{h,k}\bm{u}-\nabla\times\bm{u} ), \bm{n}\times(\nabla \times\bm{v}_h-\widehat{\bm d}_h )  \rangle_{\partial\mathcal{T}_h},\nonumber\\
	E_4=&-\langle \widehat{q}_h-q_h,(\bm{\Pi}^{{\rm curl}}_{h,k}\bm{u}-\bm u)\cdot\bm n \rangle_{\partial\mathcal{T}_h},\nonumber\\
	E_5=&\langle h_F^{-1}( \Pi_{k}^op-p),q_h-\widehat{q}_h \rangle_{\partial\mathcal{T}_h},\nonumber\\
	E_6=&- (\nabla  q_h,\bm{\Pi}^{{\rm curl}}_{h,k}\bm{u}-\bm u)_{\mathcal T_h}.\nonumber
	\end{align*}
	Then by \eqref{eq:EJh} we have
	\begin{align*}
	E^{\mathcal{J}}_h(\bm{\sigma};\bm{\tau}_h)=E_1+E_2+E_3+E_1+E_4+E_5+E_6,
	\end{align*}
	and we will bound each $E_j\,(1\leq j\leq 6)$ separately.
	It follows from the definition of $E_1$ and the fact $\langle\bm n\times\widehat{\bm v}_h, \nabla\times\bm r \rangle_{\partial\mathcal{T}_h}=0$ that
	\begin{align*}
	E_1=& -\langle \bm{n}\times\widehat{\bm{v}}_h,\nabla \times(\bm{\Pi}_{k-1}^{o}\bm{r}-\bm{r}) \rangle_{\partial\mathcal{T}_h}+\langle \bm{n}\times\bm{v}_h,\nabla \times(\bm{\Pi}_{k-1}^{o}\bm{r}-\bm{r}) \rangle_{\partial\mathcal{T}_h}\nonumber\\
	=& -\langle \bm{n}\times\widehat{\bm{v}}_h,\nabla \times\bm{\Pi}_{k-1}^{o}\bm{r} \rangle_{\partial\mathcal{T}_h}+\langle \bm{n}\times\bm{v}_h,\nabla \times(\bm{\Pi}_{k-1}^{o}\bm{r}-\bm{r}) \rangle_{\partial\mathcal{T}_h}.
	\end{align*}
	By \eqref{ia}, \eqref{ib} and integration by parts we have
	\begin{align*}
	E_1=& -\langle \bm{n}\times\bm{\Pi}_h(\bm v_h,\widehat{\bm{v}}_h),\nabla \times\bm{\Pi}_{k-1}^{o}\bm{r} \rangle_{\partial\mathcal{T}_h}\nonumber\\
	&\quad-( \bm{v}_h,\nabla \times\nabla \times(\bm{\Pi}_{k-1}^{o}\bm{r}-\bm{r}) )_{\mathcal T_h}+( \nabla \times\bm{v}_h,\nabla \times(\bm{\Pi}_{k-1}^{o}\bm{r}-\bm{r}) )_{\mathcal T_h}\nonumber\\
	=&-\langle \bm{n}\times\bm{\Pi}_h(\bm v_h,\widehat{\bm{v}}_h),\nabla \times\bm{\Pi}_{k-1}^{o}\bm{r} \rangle_{\partial\mathcal{T}_h}-(\bm{v}_h,\nabla \times\nabla \times\bm{\Pi}_{k-1}^{o}\bm{r} )_{\mathcal T_h}\nonumber\\
	&\quad+( \bm{v}_h,\nabla\times\nabla\times\bm{r} )_{\mathcal T_h}+( \nabla \times\bm{v}_h,\nabla \times(\bm{\Pi}_{k-1}^{o}\bm{r}-\bm{r}) )_{\mathcal T_h}\nonumber\\
	=& -\langle \bm{n}\times\bm{\Pi}_h(\bm v_h,\widehat{\bm{v}}_h),\nabla \times\bm{\Pi}_{k-1}^{o}\bm{r} \rangle_{\partial\mathcal{T}_h}+( \bm{\Pi}_h(\bm{v}_h,\widehat{\bm v}_h),\nabla \times\nabla \times\bm{\Pi}_{k-1}^{o}\bm{r} )_{\mathcal T_h}\nonumber\\
	&\quad+( \bm{v}_h,\nabla\times\nabla\times\bm{r} )_{\mathcal T_h}+( \nabla \times\bm{v}_h,\nabla \times(\bm{\Pi}_{k-1}^{o}\bm{r}-\bm{r}) )_{\mathcal T_h}\nonumber\\
	=&-( \nabla\times\bm{\Pi}_h(\bm{v}_h,\widehat{\bm v}_h),\nabla \times\bm{\Pi}_{k-1}^{o}\bm{r} )_{\mathcal T_h}\nonumber\\
	&\quad+( \bm{v}_h,\nabla\times\nabla\times\bm{r} )_{\mathcal T_h}+( \nabla \times\bm{v}_h,\nabla \times(\bm{\Pi}_{k-1}^{o}\bm{r}-\bm{r}) )_{\mathcal T_h}.
	\end{align*}
	Due to the fact $(\nabla\times\bm{\Pi}_h(\bm v_h,\widehat{\bm v}_h),\nabla\times\bm r)_{\mathcal T_h}-(\bm{\Pi}_h(\bm v_h,\widehat{\bm v}_h),\nabla\times\nabla\times\bm r)_{\mathcal T_h}=0$, we have
	\begin{align}\label{eq:E1Est1}
	E_1=& -( \nabla\times\bm{\Pi}_h(\bm{v}_h,\widehat{\bm v}_h),\nabla \times(\bm{\Pi}_{k-1}^{o}\bm{r}-\bm r) )_{\mathcal T_h}\nonumber\\
	&\quad+( \bm{v}_h-\bm{\Pi}_h(\bm{v}_h,\widehat{\bm v}_h),\nabla\times\nabla\times\bm{r} )_{\mathcal T_h}+( \nabla \times\bm{v}_h,\nabla \times(\bm{\Pi}_{k-1}^{o}\bm{r}-\bm{r}) )_{\mathcal T_h}\nonumber\\
	=&( \nabla \times\bm{v}_h- \nabla\times\bm{\Pi}_h(\bm{v}_h,\widehat{\bm v}_h),\nabla \times(\bm{\Pi}_{k-1}^{o}\bm{r}-\bm r ))_{\mathcal T_h}\nonumber\\
	&\quad+( \bm{v}_h-\bm{\Pi}_h(\bm{v}_h,\widehat{\bm v}_h),\nabla\times\nabla\times\bm{r} )_{\mathcal T_h}.
	\end{align}
	Note that we can rewrite the first term on the right-hand side of \eqref{eq:E1Est1} as
	\begin{align}\label{eq:E1Est2}
	\begin{split}
	&( \nabla \times\bm{v}_h- \nabla\times\bm{\Pi}_h(\bm{v}_h,\widehat{\bm v}_h),\nabla \times(\bm{\Pi}_{k-1}^{o}\bm{r}-\bm r ))_{\mathcal T_h}\\
	&\qquad\qquad=	( \nabla \times\bm{v}_h- \nabla\times\bm{\Pi}_h(\bm{v}_h,\widehat{\bm v}_h),\nabla \times\bm{\Pi}_{k-1}^{o}\bm{r}-\bm{\Pi}_{h,k+2}^{{\rm div}}\nabla\times\bm r )_{\mathcal T_h}\\
	&\qquad\qquad=	( \nabla \times\bm{v}_h- \nabla\times\bm{\Pi}_h(\bm{v}_h,\widehat{\bm v}_h),\nabla \times\bm{\Pi}_{k-1}^{o}\bm{r}-\nabla\times
	\bm{\Pi}_{h,k+2}^{{\rm curl}}\bm r )_{\mathcal T_h}.
	\end{split}
	\end{align}
	Therefore by \eqref{15_error}, \eqref{2_error}, \eqref{eq:E1Est1}, \eqref{eq:E1Est2} and inverse inequality, we get
	\begin{align}
	|E_1|&\lesssim (h^2\|\nabla\times\nabla\times\bm r\|_{_{\mathcal T_h}}+h^{s}\|\bm r\|_{s})\|h_F^{-3/2}\bm n\times(\bm v_h-\widehat{\bm v}_h)\|_{\partial\mathcal{T}_h}.
	\end{align}
	We can bound the other $E_j$ terms as follows.
	\begin{align}
	|E_2|&\lesssim \sum_{T\in\mathcal{T}_h}\|h_F^{-1/2}\bm{n}\times(\widehat{\bm d}_h-\nabla \times\bm{v}_h)\|_{\partial T}h_T^{s}\|\bm{r}\|_{s,T}\nonumber\\
	&\lesssim h^{s}\|\bm{r}\|_{s}\|h_F^{-1/2}\bm{n}\times(\widehat{\bm d}_h-\nabla \times\bm{v}_h)\|_{\partial \mathcal{T}_h},\nonumber\\
	|E_3|&\lesssim \sum_{T\in\mathcal{T}_h}\| h_F^{-1/2}\bm{n}\times(\widehat{\bm d}_h-\nabla \times\bm{v}_h)\|_{\partial T}h_T^{\min(k-1,s)}\|\nabla\times\bm{u}\|_{1+s,T}\nonumber\\
	&\lesssim h^{\min(k-1,s)}\|\nabla\times\bm u\|_{1+s}\|h_F^{-1/2}\bm{n}\times(\widehat{\bm d}_h-\nabla \times\bm{v}_h)\|_{\partial \mathcal{T}_h},\nonumber\\
	|E_4|&\lesssim h^{s}\|\bm{u}\|_{s}\| h_F^{-1/2}(q_h-\widehat{q}_h)\|_{\partial \mathcal{T}_h}\nonumber\\
	|E_5|&\lesssim h^{s}\|p\|_{s+1}\|h_F^{-1/2}(q_h-\widehat{q}_h)\|_{\partial\mathcal{T}_h}\nonumber,\\
	|E_6|&\lesssim h^{s}\|\bm u\|_{s}\|\nabla  q_h\|_{_{\mathcal T_h}} .\nonumber
	\end{align}
	This completes the proof.
\end{proof}

\begin{lemma} \label{l53}  Let $(\bm r,\bm u,p)$ be the solution of \eqref{mix0source}, then there holds
	\begin{align}\label{eq:lem427}
	\|\bm{\sigma}_h-\bm{\mathcal{J}}_h\bm{\sigma}\|_{\bm{\Sigma}_h}\lesssim& h^2\|\nabla\times\nabla\times \bm r\|_{_{\mathcal T_h}}+h^{\min(s,k-1)}
	\|\nabla\times\bm{u}\|_{s+1}\nonumber\\
	&+h^{s}(
	\|\bm{r}\|_{s}+\|\bm u\|_{s}+\|p\|_{s+1}).
	\end{align}
\end{lemma}
\begin{proof}
	By \eqref{Bh_HDG}, \Cref{Th45} and \Cref{Jh} we get
	\begin{align}\label{eq:sigmaJ}
	\begin{split}
	\|\bm{\sigma}_h-\bm{\mathcal{J}}_h\bm{\sigma}\|_{\bm{\Sigma}_h}&\lesssim\sup_{\bm{\tau}_h\in\bm{\Sigma}^{\bm{0}}_h,\bm{\tau}_h\neq\bm{0}}\frac{B_h( \bm{\sigma}_h-{\bm{\mathcal{J}}_h}\bm{\sigma},\bm{\tau}_h  )}{\|\bm{\tau}_h\|_{\bm{\Sigma}_h}}\\
	&=\sup_{\bm{\tau}_h\in\bm{\Sigma}^{\bm{0}}_h,\bm{\tau}_h\neq\bm{0}}\frac{E^{\mathcal{J}}_h( \bm{\sigma}_h,\bm{\tau}_h  )}{\|\bm{\tau}_h\|_{\bm{\Sigma}_h}}.
	\end{split}
	\end{align}
	Then \eqref{eq:lem427} directly follows from \eqref{eq:tauh}, \eqref{norm-gamma}, \eqref{eq:sigmaJ} and \Cref{lemma53}.
\end{proof}
%
%
%

\begin{theorem}\label{Thm:r} Let $(\bm r,\bm u,p)$ be the solution of \eqref{mix0source}, then we have
	\begin{align*}
	&\|\nabla\times(\bm u-\bm u_h)\|_{_{\mathcal T_h}}+
	\|\bm r-\bm r_h\|_{_{\mathcal T_h}}+\|\nabla(p-p_h)\|_{_{\mathcal T_h}}\nonumber\\
	&\qquad\lesssim h^2\|\nabla\times\nabla\times \bm r\|_{_{\mathcal T_h}}+h^{\min(s,k-1)}
	\|\nabla\times\bm{u}\|_{s+1}+h^{s}(
	\|\bm{r}\|_{s}+\|\bm u\|_{s}+\|p\|_{s+1}).
	\end{align*}
\end{theorem}
%

%
%
\subsection{Error estimates by dual arguments}

We assume $\bm{\Theta}\in \bm H({\rm div}^0;\Omega)$ and introduce the problem:
\begin{eqnarray}
\left\{
\begin{aligned}
\bm{r}^d- \nabla\times \nabla\times \bm{u}^d&=0 \,\,\,\,\,\text{ in }\Omega,\\
\nabla\times \nabla\times\bm{r}^d+\nabla p^d&=\bm{\Theta}\,\,\, \text{ in }\Omega, \label{dual}\\
\nabla\cdot\bm{u}^d&=\Lambda \,\,\,\, \text{ in }\Omega,  \\
\bm{n}_{\Gamma}\times \bm{u}^d &=\bm{0}\quad\, \text{on}\Gamma,\\
\bm{n}_{\Gamma}\times\nabla\times \bm{u}^d &=\bm{0} \,\,\,\,\, \text{ on }\Gamma,\\
p^d&=0\,\,\,\,\,  \text{ on }\Gamma.
\end{aligned}
\right.
\end{eqnarray}
Assume that
\begin{align}
\|\bm{r}^d\|_{\alpha}+\|\bm{u}^d\|_{1+\alpha,{\rm curl}}\lesssim \|\bm{\Theta}\|_{_{\mathcal T_h}}+\|\Lambda\|_{_{\mathcal T_h}},\label{reg}
\end{align}
where $\alpha\in (\frac{1}{2},1]$ is dependent on $\Omega$. It is obviously that $p^d=0$. Note that when $\Omega$ is convex, \eqref{reg} holds with $\alpha=1$.

\begin{lemma}\label{lemma55} Let $\bm{\sigma}$ and $\bm{\sigma}^d$ be the solutions of \eqref{mix0source} and \eqref{dual}, respectively. We have
	\begin{align*}
	|E^{\mathcal{J}}_h(\bm{\sigma}^d; {\bm{\mathcal{J}}_h}\bm{\sigma}-\bm{\sigma}_h)|&\lesssim  h^{\min(\alpha,k-1)}(\|\bm{\Theta}\|_{_{\mathcal T_h}}+\|\Lambda\|_{_{\mathcal T_h}})\|{\bm{\mathcal{J}}_h}\bm{\sigma}-\bm{\sigma}_h\|_{\bm{\Sigma}_h},\nonumber\\
	|E^{\mathcal{J}}_h(\bm{\sigma};\bm{\mathcal{J}}_h\bm{\sigma}^d)| &\lesssim 
	h^{\min(\alpha,k-1)}(h^s\|\bm r\|_{s}
	+h^{\min(k-1,s)}\|\nabla\times\bm{u}\|_{s+1})
	(\|\bm{\Theta}\|_{_{\mathcal T_h}}+\|\Lambda\|_{_{\mathcal T_h}}).
	\end{align*}
\end{lemma}
\begin{proof}
	Similar to the proof of \Cref{lemma53}, we get	
	\begin{align*}
	|E^{\mathcal{J}}_h(\bm{\sigma}^d; {\bm{\mathcal{J}}_h}\bm{\sigma}-\bm{\sigma}_h)|\lesssim&  \left(h^2\|\nabla\times\nabla\times \bm r^d\|_{_{\mathcal T_h}}+h^{\min(\alpha,k-1)}\|\nabla\times\bm{u}^d\|_{1+\alpha}\right)\|{\bm{\mathcal{J}}_h}\bm{\sigma}-\bm{\sigma}_h\|_{\bm{\Sigma}_h}\nonumber\\
	&\quad+h^{\alpha}\left(\|\bm{r}^d\|_{\alpha}+\|\bm u^d\|_{1+\alpha}\right)\|{\bm{\mathcal{J}}_h}\bm{\sigma}-\bm{\sigma}_h\|_{\bm{\Sigma}_h}\nonumber\\
	\lesssim& h^{\min(\alpha,k-1)}(\|\bm{\Theta}\|_{_{\mathcal T_h}}+\|\Lambda\|_{_{\mathcal T_h}})\|{\bm{\mathcal{J}}_h}\bm{\sigma}-\bm{\sigma}_h\|_{\bm{\Sigma}_h},
	\end{align*}
	%
	%
	By noticing $p^d=0$, it holds
	\begin{align*}
	E^{\mathcal{J}}_h(\bm{\sigma};\bm{\mathcal{J}}_h\bm{\sigma}^d)&\lesssim (h^s\|\bm r\|_{s}
	+h^{\min(k-1,s)}\|\nabla\times\bm{u}\|_{s+1})
	\|h_F^{-1/2}\bm{n}\times(\widehat{\bm d}_h-\nabla \times\bm{v}_h)\|_{0,\partial \mathcal{T}_h}\nonumber\\
	&\lesssim 
	h^{\min(\alpha,k-1)}(h^s\|\bm r\|_{s}
	+h^{\min(k-1,s)}\|\nabla\times\bm{u}\|_{s+1})
	(\|\bm{\Theta}\|_{_{\mathcal T_h}}+\|\Lambda\|_{_{\mathcal T_h}}).
	\end{align*}
\end{proof}

\begin{theorem}\label{Thm:L^2}
	Let $(\bm r,\bm u,p)$ and $(\bm r_h,\bm u_h,\widehat{\bm u}_h,p_h,\widehat{p}_h)$ be the solutions of \eqref{mix0source} and \eqref{Bh_HDG}, respectively, then there holds

	\begin{align}\label{eq:L2err}
	\begin{split}
	&\|\bm u-\bm u_h\|_{_{\mathcal T_h}}+\|p-p_h\|_{_{\mathcal T_h}}\\
	&\qquad\lesssim 
	h^{\min(\alpha,k-1)}
	\left(
	h^2\|\nabla\times\nabla\times \bm r\|_{_{\mathcal T_h}}+h^{\min(s,k-1)}
	\|\nabla\times\bm{u}\|_{s+1}
	\right)\\
	&\qquad\quad+
	h^{s+\min(\alpha,k-1)}
	\left(
	\|\bm{r}\|_{s}+\|\bm u\|_{s}+\|p\|_{s+1}
	\right)+
	\|\bm u-\bm{\Pi}_{h,k}^{\rm curl}\bm u\|_{_{\mathcal T_h}}.
	\end{split}
	\end{align}

\end{theorem}
\begin{proof}
	We introduce a projection $\bm{\Pi}_k^{\rm m}$. For all $\bm v\in \bm H^s({\rm curl};\Omega)$ with $s>1/2$ and $\bm v_h\in \bm U_h$, such that
	\begin{align}\label{m1}
	\bm{\Pi}_k^{\rm m}(\bm v,\bm v_h)=\bm{\Pi}_{h,k}^{\rm curl}\bm v+\nabla \sigma_h,
	\end{align}
	where $\sigma_h\in \mathcal P_k(\mathcal T_h)\cap H^1_0(\Omega)$ satisfies 
	\begin{align}\label{m2}
	(\nabla\sigma_h,\nabla q_h)_{\mathcal T_h}=(\bm{\Pi}_{h,k}^{\rm curl,c}(\bm v_h-\bm{\Pi}_{h,k}^{\rm curl}\bm v),\nabla q_h)_{\mathcal T_h}\qquad \forall q_h\in \mathcal P_k(\mathcal T_h)\cap H^1_0(\Omega).
	\end{align}
	From \eqref{m1} and \eqref{m2}, it holds
	\begin{align}\label{oror}
	(\bm{\Pi}_{h,k}^{\rm curl,c}(\bm v_h-\bm{\Pi}_k^{\rm m}(\bm v,\bm v_h)),\nabla q_h)=0\qquad \forall q_h\in \mathcal P_k(\mathcal T_h)\cap H^1_0(\Omega).
	\end{align}
	We take
	$\Lambda=\Pi_k^op-p_h$ in \eqref{dual} and let $\bm\Theta\in \bm H({\rm curl};\Omega)\cap \bm H({\rm div};\Omega)$ be the solution of
	\begin{subequations}
		\begin{align*}
		\nabla\times\bm\Theta&=\nabla\times ( \bm{\Pi}_{h,k}^{\rm curl,c}(\bm u_h-\bm{\Pi}_k^{\rm m}(\bm u,\bm u_h)))&\text{in }\Omega,\\
		\nabla\cdot\bm\Theta&=0&\text{in }\Omega,\\
		\bm n\times\bm\Theta&=\bm 0&\text{on }\Gamma.
		\end{align*}
	\end{subequations}
	Due to \eqref{oror} and the result in \cite[Lemma 4.5]{MR2009375} one has
	\begin{align}
	\|\bm\Theta-( \bm{\Pi}_{h,k}^{\rm curl,c}(\bm u_h-\bm{\Pi}_k^{\rm m}(\bm u,\bm u_h) ))\|_{_{\mathcal T_h}}&\lesssim h^{\alpha}\|\nabla\times ( \bm{\Pi}_{h,k}^{\rm curl,c}(\bm u_h-\bm{\Pi}_k^{\rm m}(\bm u,\bm u_h) ))\|_{_{\mathcal T_h}}.\label{co-2}
	\end{align}
	We obtain the following estimates by \eqref{m2}, \Cref{pic}, and an inverse inequality 
	\begin{align}\label{newpi0-l}
	\begin{split}
	&\|\bm{\Pi}_{h,k}^{\rm curl}\bm u-\bm{\Pi}_k^{\rm m}(\bm u,\bm u_h)\|_{_{\mathcal T_h}}\\
	&\qquad=\|\nabla\sigma_h\|_{\mathcal T_h}\\
	&\qquad\le \|\bm{\Pi}_{h,k}^{\rm curl,c}(\bm u_h-\bm{\Pi}_{h,k}^{\rm curl}\bm u)\|_{\mathcal T_h} \\
	&\qquad\le
	\|\bm{\Pi}_{h,k}^{\rm curl,c}(\bm u_h-\bm{\Pi}_{h,k}^{\rm curl}\bm u)-(\bm u_h-\bm{\Pi}_{h,k}^{\rm curl}\bm u)\|_{\mathcal T_h} 
	+\|(\bm u_h-\bm{\Pi}_{h,k}^{\rm curl}\bm u)\|_{\mathcal T_h}\\
	&\qquad\lesssim\left(\|h_F^{1/2}\bm n\times[\![\bm u_h-\bm{\Pi}^{\rm curl}_{h,k}\bm u]\!]\|_{0,\mathcal{F}_h}+ \|\bm u_h-\bm{\Pi}^{\rm curl}_{h,k}\bm u\|_{_{\mathcal T_h}}\right) \\
	&\qquad\lesssim \|\bm{\Pi}^{\rm curl}_{h,k}\bm u-\bm u_h\|_{_{\mathcal T_h}}.
	\end{split}
	\end{align}
	Similarity, we can get
	\begin{align}
	\|\bm{\Pi}_{h,k}^{\rm curl,c}(\bm u_h-\bm{\Pi}_k^{\rm m}(\bm u,\bm u_h))\|_{_{\mathcal T_h}}
	&\lesssim \|\bm{\Pi}^{\rm curl}_{h,k}\bm u-\bm u_h\|_{_{\mathcal T_h}},\label{newpi1-l}
	\end{align}
	and
	\begin{align}
	\begin{split}
	&\|\nabla\times((\bm{\Pi}_{h,k}^{\rm curl,c}\bm u_h-\bm{\Pi}_k^{\rm m}(\bm u,\bm u_h)))\|_{_{\mathcal T_h}}\\
	&\qquad\lesssim  \left( \|h_F^{-1/2}\bm n\times[\![\bm u_h-\bm{\Pi}^{\rm curl}_{h,k}\bm u]\!]\|_{0,\mathcal{F}_h} + \|\nabla \times(\bm{\Pi}^{\rm curl}_{h,k}\bm u-\bm u_h)\|_{_{\mathcal T_h}} \right)\\
	&\qquad\lesssim \left(\|\bm\sigma-\bm\sigma_h\|_{\bm\Sigma_h}+\|\nabla \times(\bm{\Pi}^{\rm curl}_{h,k}\bm u-\bm u_h)\|_{_{\mathcal T_h}}\right)\\
	&\qquad\lesssim \left(\|\bm\sigma-\bm\sigma_h\|_{\bm\Sigma_h}+h^{s}\|\bm r\|_s\right)\\
	&\qquad\lesssim h^{s}(\|\bm{r}\|_{s}+\|\bm{u}\|_{s}+\|p\|_{s+1} ).\label{newpi2-l}
	\end{split}
	\end{align}
	It then follows from \eqref{co-2}  and \eqref{newpi2-l} that
	\begin{align}
	\|\bm\Theta-( \bm{\Pi}_{h,k}^{\rm curl,c}(\bm u_h-\bm{\Pi}_k^{\rm m}(\bm u,\bm u_h) ))\|_{_{\mathcal T_h}}&\lesssim h^{s+\alpha}(\|\bm{r}\|_{s}+\|\bm{u}\|_{s}+\|p\|_{s+1} ).\label{theta-l-1}
	\end{align}	
	Follows from the above estimates inequality, it holds that
	\begin{align}
	\|\bm\Theta\|_{_{\mathcal T_h}}&\le
	\|\bm\Theta-( \bm{\Pi}_{h,k}^{\rm curl,c}(\bm u_h-\bm{\Pi}_k^{\rm m}(\bm u,\bm u_h) ))\|_{_{\mathcal T_h}}
	+\| \bm{\Pi}_{h,k}^{\rm curl,c}(\bm u_h-\bm{\Pi}_k^{\rm m}(\bm u,\bm u_h) )\|_{_{\mathcal T_h}}\nonumber\\
	&\lesssim h^{s+\alpha}(\|\bm{r}\|_{s}+\|\bm{u}\|_{s}+\|p\|_{s+1} )+\|\bm{\Pi}_{h,k}^{\rm curl}\bm u-\bm u_h\|_{_{\mathcal T_h}}.\label{theta-l-2}
	\end{align}
	In view of \Cref{Jh}, we have
	\begin{align}
	B_h(\bm{\mathcal{J}}_h\bm{\sigma}^d,\bm{\tau}_h)=-(\bm\Theta,\bm{v}_h)_{\mathcal T_h}+(\Lambda,q_h)_{\mathcal T_h}+E^{\mathcal{J}}_h(\bm{\sigma}^d;\bm{\tau}_h)\quad\forall\, \bm{\tau}_h\in\bm{\Sigma}_h^{\bm 0}. \label{b2}
	\end{align}
	We take $\bm{\tau}_h=\bm{\mathcal J}_h\bm{\sigma}-\bm{\sigma}_h\in \bm\Sigma_h^{\bm 0}$ in \eqref{b2},  and use \eqref{510}, \eqref{Bh_HDG} to get
	\begin{align}\label{B1}
	\begin{split}
	&-(\bm\Theta,\bm{\Pi}^{\rm curl}_{h,k}\bm u-\bm u_h)_{\mathcal T_h}+\|\Pi_k^op-p_h\|_{\mathcal T_h}^2\\
	&\qquad= B_h(\bm{\mathcal{J}}_h\bm{\sigma}^d,\bm{\mathcal J}_h\bm{\sigma}-\bm{\sigma}_h)
	-E^{\mathcal{J}}_h(\bm{\sigma}^d;\bm{\mathcal J}_h\bm{\sigma}-\bm{\sigma}_h)
	\\
	&\qquad= B_h(\bm{\mathcal J}_h\bm{\sigma}-\bm{\sigma}_h,\bm{\mathcal{J}}_h\bm{\sigma}^d)
	-E^{\mathcal{J}}_h(\bm{\sigma}^d;\bm{\mathcal J}_h\bm{\sigma}-\bm{\sigma}_h)
	\\
	&\qquad= E^{\mathcal{J}}_h(\bm{\sigma};\bm{\mathcal{J}}_h\bm{\sigma}^d)
	-E^{\mathcal{J}}_h(\bm{\sigma}^d;\bm{\mathcal J}_h\bm{\sigma}-\bm{\sigma}_h).
	\end{split}
	\end{align}
	We take $q_h=\widehat q_h=\sigma_h$ in \eqref{fhm3} to get
	\begin{align}\label{on}
	-(\bm u_h,\nabla\sigma_h)_{\mathcal T_h}=(g,\sigma_h)_{\mathcal T_h}.
	\end{align}
	We use a direct calculation to get
	\begin{align}
	&(\bm u-\bm{\Pi}_k^{\rm m}(\bm u,\bm u_h),\bm u-\bm u_h   )_{\mathcal T_h}\nonumber\\
	&\qquad=
	(\bm u-\bm{\Pi}_{h,k}^{\rm curl}\bm u,\bm u-\bm u_h)_{\mathcal T_h}
	+
	(-\nabla\sigma_h,\bm u-\bm u_h   )_{\mathcal T_h}
	&\text{by the definiton of $\bm\Pi^{\rm m}_k$}
	\nonumber\\
	&\qquad=
	(\bm u-\bm{\Pi}_{h,k}^{\rm curl}\bm u,\bm u-\bm u_h)_{\mathcal T_h}
	+
	(\sigma_h,\nabla\cdot\bm u   )_{\mathcal T_h}
	+(\nabla\sigma_h,\bm u_h)_{\mathcal T_h} 
	&\text{by integration by parts}
	\nonumber\\
	&\qquad=(\bm u-\bm{\Pi}_{h,k}^{\rm curl}\bm u,\bm u-\bm u_h)_{\mathcal T_h}
	&\text{by }\eqref{source}, \eqref{on} \label{444}.\nonumber\\
	\end{align}
	We use \eqref{ic1} to get
	\begin{align}\label{icc}
	\begin{split}
	&\|\bm{\Pi}_{h,k}^{\rm curl,c}(\bm u_h-\bm{\Pi}_k^{\rm m}(\bm u,\bm u_h) )-(\bm u_h-\bm{\Pi}_k^{\rm m}(\bm u,\bm u_h)\|_{_{\mathcal T_h}}\\
	&\qquad\lesssim h\|h_F^{-1/2}\bm n\times[\![\bm u_h-\bm{\Pi}_k^{\rm m}(\bm u,\bm u_h) ]\!]\|_{0,\mathcal F_h}\\
	&\qquad= h\|h_F^{-1/2}\bm n\times[\![\bm u_h ]\!]\|_{0,\mathcal F_h^I}
	+
	Ch\|h_F^{-1/2}\bm n\times[\![\bm u_h-\bm{\Pi}_{h,k}^{\rm curl}\bm u ]\!]\|_{0,\mathcal F_h^B}
	\\
	&\qquad= h\|h_F^{-1/2}\bm n\times[\![\bm u_h-\widehat{\bm u}_h ]\!]\|_{0,\mathcal F_h^I}
	+
	h\|h_F^{-1/2}\bm n\times[\![\bm u_h-\widehat{\bm u}_h]\!]\|_{0,\mathcal F_h^B}
	\\
	&\qquad\lesssim
	h\|\bm\sigma_h-\bm{\mathcal I}_h\bm\sigma\|_{\bm{\Sigma}_h} \\
	&\qquad\lesssim   h^{s+1}(\|\bm{r}\|_{s}+\|\bm{u}\|_{s}+\|p\|_{s+1} ).
	\end{split}
	\end{align}
	We use a direct calculation to get
	\begin{align}\label{4444}
	&(\bm{\Pi}_{h,k}^{\rm curl}\bm u-\bm{\Pi}_k^{\rm m}(\bm u,\bm u_h),\bm{\Pi}_{h,k}^{\rm curl}\bm u-\bm u_h   )\nonumber\\
	&\qquad=
	(\bm{\Pi}_{h,k}^{\rm curl}\bm u-\bm{\Pi}_k^{\rm m}(\bm u,\bm u_h),\bm u-\bm u_h   )
	\nonumber\\
	&\qquad\quad+(\bm{\Pi}_{h,k}^{\rm curl}\bm u-\bm{\Pi}_k^{\rm m}(\bm u,\bm u_h),\bm{\Pi}_{h,k}^{\rm curl}\bm u-\bm u   )\nonumber\\
	&\qquad=(-\nabla\sigma_h,\bm u-\bm u_h   )
	+(\bm{\Pi}_{h,k}^{\rm curl}\bm u-\bm{\Pi}_k^{\rm m}(\bm u,\bm u_h),\bm{\Pi}_{h,k}^{\rm curl}\bm u-\bm u   )
	&\text{by the definiton of $\bm\Pi^{\rm m}_k$}
	\nonumber\\
	&\qquad=(\sigma_h,\nabla\cdot\bm u   )
	+(\nabla\sigma_h,\bm u_h) \nonumber\\
	&\qquad\quad
	+(\bm{\Pi}_{h,k}^{\rm curl}\bm u-\bm{\Pi}_k^{\rm m}(\bm u,\bm u_h),\bm{\Pi}_{h,k}^{\rm curl}\bm u-\bm u   )
	&\text{by integration by parts}
	\nonumber\\
	&\qquad=(\bm{\Pi}_{h,k}^{\rm curl}\bm u-\bm{\Pi}_k^{\rm m}(\bm u,\bm u_h),\bm{\Pi}_{h,k}^{\rm curl}\bm u-\bm u   )
	&\text{by }\eqref{source}, \eqref{on}
	\nonumber\\
	&\qquad\le C\|\bm{\Pi}_{h,k}^{\rm curl}\bm u-\bm u_h\|_{_{\mathcal T_h}}\|\bm{\Pi}_{h,k}^{\rm curl}\bm u-\bm u \|_{_{\mathcal T_h}}
	&\text{by }\eqref{newpi0-l}. \nonumber\\
	\end{align}
	By using \eqref{B1}, one can obtain
	\begin{align*}
	\|\bm{\Pi}^{\rm curl}_{h,k}\bm u-\bm u_h\|_{\mathcal T_h}^2
	&=(\bm{\Pi}^{\rm curl}_{h,k}\bm u-\bm u_h,\bm{\Pi}^{\rm curl}_{h,k}\bm u-\bm u_h   )_{\mathcal T_h}\nonumber\\
	&=-(\bm\Theta,\bm{\Pi}^{\rm curl}_{h,k}\bm u-\bm u_h   )_{\mathcal T_h}+(\bm\Theta-( \bm{\Pi}_{h,k}^{\rm curl,c}(\bm u_h-\bm{\Pi}_k^{\rm m}(\bm u,\bm u_h) )),\bm{\Pi}^{\rm curl}_{h,k}\bm u-\bm u_h   )_{\mathcal T_h}\nonumber\\
	&\quad+(\bm{\Pi}^{\rm curl}_{h,k}\bm u-\bm{\Pi}_k^{\rm m}(\bm u,\bm u_h),\bm{\Pi}^{\rm curl}_{h,k}\bm u-\bm u_h   )_{\mathcal T_h}\nonumber\\
	&\quad+(\bm{\Pi}_{h,k}^{\rm curl,c}(\bm u_h-\bm{\Pi}_k^{\rm m}(\bm u,\bm u_h) )-(\bm u_h-\bm{\Pi}_k^{\rm m}(\bm u,\bm u_h) )
	,\bm{\Pi}^{\rm curl}_{h,k}\bm u-\bm u_h)_{\mathcal T_h}\nonumber\\
	&=E^{\mathcal{J}}_h(\bm{\sigma};\bm{\mathcal{J}}_h\bm{\sigma}^d)
	-E^{\mathcal{J}}_h(\bm{\sigma}^d;\bm{\mathcal J}_h\bm{\sigma}-\bm{\sigma}_h)
	-\|\Pi_k^op-p_h\|_{\mathcal T_h}^2\nonumber\\
	&\quad+(\bm\Theta-( \bm{\Pi}_{h,k}^{\rm curl,c}(\bm u_h-\bm{\Pi}_k^{\rm m}(\bm u,\bm u_h) )),\bm{\Pi}^{\rm curl}_{h,k}\bm u-\bm u_h   )_{\mathcal T_h}\nonumber\\
	&\quad+(\bm{\Pi}^{\rm curl}_{h,k}\bm u-\bm{\Pi}_k^{\rm m}(\bm u,\bm u_h),\bm{\Pi}^{\rm curl}_{h,k}\bm u-\bm u_h   )_{\mathcal T_h}\nonumber\\
	&\quad+(\bm{\Pi}_{h,k}^{\rm curl,c}(\bm u_h-\bm{\Pi}_k^{\rm m}(\bm u,\bm u_h) )-(\bm u_h-\bm{\Pi}_k^{\rm m}(\bm u,\bm u_h) )
	,\bm{\Pi}^{\rm curl}_{h,k}\bm u-\bm u_h)_{\mathcal T_h}\nonumber\\
	\end{align*}
	which together with \Cref{lemma55}, \eqref{theta-l-1}, \eqref{icc} and \eqref{4444} implies \eqref{eq:L2err}. 
	
\end{proof}

\section{Numerical experiments}
All numerical tests in this section are programmed in C++. When implementing the HDG method \eqref{fhm1}--\eqref{fhm3}, all the interior unknowns $\bm{r}_h$, ${\bm{u}_h}$ and ${p_h}$ are eliminated. The only global unknowns of the resulting system are $\widehat{\bm{u}}_h$, $\widehat{\bm c}_h$ and $\widehat{p}_h$; and then $\bm{r}_h$, ${\bm{u}_h}$ and ${p_h}$ can be recovered locally. This is the unique feature of HDG method. The solver for the linear system is chosen as GMRES, which uses AMG as preconditioner. We take $\mathcal{T}_h$ to be a uniform simplex decomposition of $\Omega$ in all examples.

\subsection{Smooth case}
We take $\Omega=(0,1)^3$. The functions
$\bm{r}$, $\bm{f}$, $g$ and $\bm{g}_T$ are determined according to the following true solutions
\begin{align*}
u_1=\sin(y)\sin(z),\
u_2=\sin(z)\sin(x),\
u_3=\sin(x)\sin(y),\
p=0.
\end{align*}
The $L_2$ errors are reported in \Cref{tab1} and \Cref{tab2} for $k=1$ and $k=2$, respectively. According to \Cref{Thm:r} and  \Cref{Thm:L^2}, we would have
\begin{align*}
\|\bm u-\bm u_h\|_{_{\mathcal T_h}}+\|\bm r-\bm r_h\|_{_{\mathcal T_h}}+\|p-p_h\|_{_{\mathcal T_h}}&\le C&k=1,\\
\|\bm u-\bm u_h\|_{_{\mathcal T_h}}+h\|\bm r-\bm r_h\|_{_{\mathcal T_h}}+\|p-p_h\|_{_{\mathcal T_h}}&\le Ch^{2}&k=2.
\end{align*}
It can be observed that the orders of convergence are better than predicted. This may due to the fact that the exact solution has high smoothness. Actually, when the true solution is smooth enough, one may derive error analysis of HDG method for the quad-curl problem similarly to the biharmonic problem and obtain better convergence rates (probably optimal with respect to $k$ for different stabilization parameters). This will be our future work. 

\begin{table}[H]
	\small
	\caption{\label{tab1}Results for $k=1$}
	\centering
	
	\begin{tabular}{c|c|c|c|c|c|c|c}
		\Xhline{1pt}
		\multirow{2}{*}{$h^{-1}$} &
		\multicolumn{2}{c|}{$\|\bm{r}-\bm{r}_h\|_{_{\mathcal T_h}}/\|\bm r\|_{_{\mathcal T_h}}$} &
		\multicolumn{2}{c|}{$\|\bm{u}-\bm{u}_h\|_{_{\mathcal T_h}}/\|\bm u\|_{_{\mathcal T_h}}$} &
		
		\multicolumn{2}{c| }{$\|p-p_{h}\|_{_{\mathcal T_h}}$} &
		\multirow{2}{*}{DOF}\\
		\cline{2-7}
		&Error &Rate  &Error &Rate  &Error &Rate \\
		\hline
		2	&3.57E-01&	    &1.34E-01	&	    &1.05E-02	&	   &1320\\
		4	&1.92E-01	&0.93 	&3.42E-02	&1.97 	&1.64E-03	&2.69 &9508\\
		8	&1.02E-01	&0.91 	&8.67E-03	&1.98 	&2.13E-04	&2.94 &71808\\
		16	&5.72E-02	&0.83 	&2.20E-03	&1.98 	&2.70E-05	&2.98 	&557568\\
		
		\Xhline{1pt}
	\end{tabular}
\end{table}

\begin{table}[H]
	\small
	\caption{\label{tab2}Results for $k=2$}
	\centering
	
	\begin{tabular}{c|c|c|c|c|c|c|c}
		\Xhline{1pt}
		\multirow{2}{*}{$h^{-1}$} &
		\multicolumn{2}{c|}{$\|\bm{r}-\bm{r}_h\|_{_{\mathcal T_h}}/\|\bm r\|_{_{\mathcal T_h}}$} &
		\multicolumn{2}{c|}{$\|\bm{u}-\bm{u}_h\|_{_{\mathcal T_h}}/\|\bm u\|_{_{\mathcal T_h}}$} &
		\multicolumn{2}{c| }{$\|p-p_{h}\|_{_{\mathcal T_h}}$} &
		\multirow{2}{*}{DOF}\\
		\cline{2-7}
		
		&Error &Rate  &Error &Rate  &Error &Rate \\
		\hline
		
		2	&3.85E-02	&	    &2.73E-02	&	    &1.60E-03	&	    &2880\\
		4	&1.26E-02	&1.61 	&2.11E-03	&3.69 	&6.41E-05	&4.64 	&20736\\
		8	&4.90E-03	&1.37 	&1.56E-04	&3.75 	&3.40E-06	&4.24 	&156672\\
		
		\Xhline{1pt}
	\end{tabular}	
\end{table}

\subsection{Singular solution on L-shaped domain}

We take $\Omega=(-1,1)^3/(-1,0)\times(-1,0)\times(-1,1)$. The functions $\bm{f}$, $g$ and $\bm{g}_T$ are determined according to the following true solutions
\begin{align*}
u_1=tr^{t-1}\sin[(t-1)\theta],\,\,
u_2=tr^{t-1}\cos[(t-1)\theta],\,\,
u_3=0,\,\,
\bm{r}=\bm{0},\,\,
p=0.
\end{align*}

By taking $t=0.9$ and $t=1.4$, we have $\bm{u}\in [H^{0.9-\epsilon}(\Omega)]^3$ and $\bm{u}\in [H^{1.4-\epsilon}(\Omega)]^3$, respectively, for arbitrarily small $\epsilon>0$. The results for $k=1$ are reported in \Cref{tab3} and \Cref{tab4}. 
\color{black} In this case, we have $\nabla\times\bm u=\bm 0$,  therefore, by \Cref{Thm:r} and \Cref{Thm:L^2} we have
\begin{align*}
\|\bm u-\bm u_h\|_{_{\mathcal T_h}}+\|\bm r-\bm r_h\|_{_{\mathcal T_h}}+\|p-p_h\|_{_{\mathcal T_h}}&\le Ch^{t-\epsilon}\|\bm u\|_{t-\epsilon}
&t=0.9, 1.4.
\end{align*}
We observe that optimal convergence rate with respect the regularity for $\|\bm u-\bm u_h\|_{_{\mathcal T_h}}$ is obtained, which verifies the theoretical results. Moreover, the convergence rates for  $\|\bm r-\bm r_h\|_{_{\mathcal T_h}}$ and $\|p-p_h\|_{_{\mathcal T_h}}$
are  better than predicted.
\color{black}

\begin{table}[H]
	\small
	\caption{ \label{tab3}Results for $k=1$, $t=0.9$}
	\centering
	
	\begin{tabular}{c|c|c|c|c|c|c|c}
		\Xhline{1pt}
		
		\multirow{2}{*}{$h^{-1}$} &		
		\multicolumn{2}{c|}{$\|\bm{r}-\bm{r}_h\|_{_{\mathcal T_h}}$} &
		\multicolumn{2}{c|}{$\|\bm{u}-\bm{u}_h\|_{_{\mathcal T_h}}$} &				
		\multicolumn{2}{c| }{$\|p-p_{h}\|_{_{\mathcal T_h}}$} &
		\multirow{2}{*}{DOF}\\
		\cline{2-7}
		
		&Error &Rate  &Error &Rate  &Error &Rate  \\
		\hline
		2	&2.73E-03	&	    &7.35E-02	&	    &3.50E-02	&	    &1034\\
		4	&2.65E-03	&0.04 	&4.20E-02	&0.81 	&2.29E-02	&0.61 	&7304\\
		8	&1.07E-03	&1.31 	&2.36E-02	&0.83 	&8.00E-03	&1.52 	&54560\\
		16	&3.89E-04	&1.46 	&1.27E-02	&0.90 	&2.46E-03	&1.70 	&420992\\

		\Xhline{1pt}
	\end{tabular}	
\end{table}

\begin{table}[H]
	\small 
	\caption{\label{tab4}Results for $k=1$, $t=1.4$}
	\centering
	
	\begin{tabular}{c|c|c|c|c|c|c|c}
		\Xhline{1pt}
		
		\multirow{2}{*}{$h^{-1}$} &	
		\multicolumn{2}{c|}{$\|\bm{r}-\bm{r}_h\|_{_{\mathcal T_h}}$} &
		\multicolumn{2}{c|}{$\|\bm{u}-\bm{u}_h\|_{_{\mathcal T_h}}$} &	
		\multicolumn{2}{c| }{$\|p-p_{h}\|_{_{\mathcal T_h}}$} &
		\multirow{2}{*}{DOF}\\
		\cline{2-7}
		
		&Error &Rate  &Error &Rate  &Error &Rate  \\
		\hline
		2	&5.38E-03	&	    &1.49E-01	&	    &8.29E-02	&	    &1034\\
		4	&2.92E-03	&0.88 	&6.34E-02	&1.23 	&3.64E-02	&1.19 	&7304\\
		8	&7.93E-04	&1.88 	&2.59E-02	&1.29 	&8.61E-03	&2.08 	&54560\\
		16	&1.81E-04	&2.13 	&9.97E-03	&1.38 	&1.80E-03	&2.26 	&420992\\
		
		\Xhline{1pt}
	\end{tabular}	
\end{table}


\bibliographystyle{siam}
\bibliography{references}{}

\end{document}